\documentclass[11pt,a4paper]{article}

\usepackage{amsmath, amsfonts, amssymb}
\usepackage{color}

\def\be#1\ee{\begin{equation}#1\end{equation}}

\newtheorem{thm}{Theorem}
\newtheorem{lem}[thm]{Lemma}
\newtheorem{prop}[thm]{Proposition}

\newtheorem{rem}[thm]{Remark}

\DeclareMathOperator{\cov}{cov}

\def\P{{\mathbb{P}}}
\def\R{\mathbb{R}}
\def\E{\mathbb{E}\,}

\def\dd{\mbox{d}}

\newenvironment{proof}[1][] {\noindent {\bf Proof#1:} }{\hspace*{\fill}$\square$\medskip\par}

\newcommand{\eps}{\varepsilon}

\def\NN{{\mathcal N}}

\def\WW{{\mathcal{W}}}

\def\tV{{\widetilde{V}}}
\def\ttau{{\widetilde{\tau}}}

\def\MA{{\mathcal{A}}}
\def\tU{{\widetilde{U}}}
\def\tV{{\widetilde{V}}}
\def\VVV{{ \Upsilon}}
\def\mix{{\rm mix}}
\def\ustar{{\widehat u}}

\def\covcoef{{\rm corr}}

\def\amu{\mu} 

\let\BFseries\bfseries\def\bfseries{\BFseries\mathversion{bold}} 

\def\tod{{~\stackrel{d}{\to}~}}     
\def\span{\overline{{\rm span}}}

\begin{document}

\title{Breaking a chain of interacting Brownian particles}
\author{Frank Aurzada, Volker Betz, and Mikhail Lifshits}
\maketitle

\begin{abstract}
We investigate the behaviour of a finite chain of Brownian particles,
interacting through a pairwise quadratic potential, with one end of the chain fixed and the other end  pulled away at slow speed,
in the limit of slow speed and small Brownian noise.

We study the instant when the chain ``breaks'', that is, the distance between two neighboring
particles becomes larger than a certain limit.  There are three different regimes depending
on the relation between the speed of pulling and the Brownian noise. We prove weak limit theorems for
the break time and the break position for each regime.

\end{abstract}

\medskip



\newpage


\tableofcontents

\newpage

\section{Introduction and main results}
\subsection{Introduction}

Interacting Brownian particles are a popular model for various
physical systems where a (possibly large) number of particles is
subjected to inter-particle forces and ambient noise. An important
example are ensembles of colloidal particles, i.e. mesoscopic
particles suspended in a fluid. Approaches to model these particles
through interacting Brownian motions go back at least to
\cite{ErCam78}, and the investigation of
increasingly complex situations continues until today, see e.g.\
\cite{SKF2016}. Since the focus of this paper is on a detailed
investigation of a specific mathematical model, we will not attempt
to give an account of the vast physics literature on the topic, but
rather refer to the introduction of \cite{SKF2016} and to the review
article \cite{Sei12}.

The specific model that we are interested in consists of
a chain of $d+1$ interacting Brownian particles, where
nearest neighbour particles are coupled by a
harmonic potential, the leftmost particle is fixed, and the
rightmost particle is slowly pulled to the right. Furthermore, we
assume that the chain breaks whenever the distance between
two neighbouring particles exceeds a certain threshold. Both the
variance of the noise $\sigma^2$  and the speed
of pulling $\eps$ are assumed to be small, and we are interested in the
position and time of breakage depending on the relative behaviour
of $\sigma$ and $\eps$ as both go to zero. We refer to equation
\eqref{eqn:original} below for the precise definition. Models of this type have been
studied in the physics literature in the context of rupture of polymer
chains \cite{FS1,FS2,lee,sain}.

Our results
are an essentially complete description of the behaviour of the
chain in the following three regimes,
which cover almost all possible behaviours of $
\sigma$ and $\eps$. We always assume that both $\sigma$ and
$\eps$ converge to zero. The case of {\it fast pulling} is characterised by the condition
$\sigma/\eps \to 0$. Then the chain behaves as if $\sigma = 0$,
i.e.\ it breaks deterministically at the $d$-th bond, see Theorem \ref{thm:quasideterm}.
On the other end of the scale,
the case of {\it very slow pulling} occurs when
$\sigma^2 |\ln \eps| \to \infty$. For $\eps = 0$, Lee
\cite{lee} made (and numerically verified) the
somewhat surprising prediction that the chain
breaks at the first and the last bond with
probability $1/(2(d-1))$ each, and at any other bond with probability
$1/(d-1)$. We prove that this prediction is indeed accurate, and
that it extends to the whole regime of very slow pulling.
Moreover, we find that the (properly rescaled) time of breakage
follows an exponential distribution, see Theorem
\ref{thm:veryslowpulling}.
 The third regime, which we call {\it
moderately slow pulling}, occurs when
$\eps/\sigma \to 0$ but $\sigma^2 |\ln \eps| \to 0$. Here we
find that in this regime, the breaking probabilities of the bonds
are just like in the regime of very slow pulling,  but the breakage
time is distributed differently: its difference from the
the deterministic time $t_\ast = d/\eps$ follows a Gumbel
distribution after proper rescaling. This shows that the mechanism
that is responsible for the the break event is slightly different in the
cases of very slow and moderately slow pulling. While large deviations events
are present in both regimes, in the case of
very slow pulling one essentially deals with an exit event for a stationary
process over a high constant level; while in the
case of moderately slow pulling one encounters an exit over
a level linearly varying in time. The appearance of the Gumbel
distribution in exit problems is not new, see
\cite{Bak15, Ber16, Day90}.
However, in all those cases it appears in the
context of conditioning on an unlikely exit, while in our case, it
appears in the context of the most likely exit. To our knowledge,
this phenomenon has not been observed before.

Mathematically, our problem falls squarely into the class of exit
problems for stochastic processes. This is a topic having a
long tradition that started with the early non-rigorous works of
Kramers \cite{Kra40} and Eyring \cite{Ey35}, was significantly
influenced by the works of Freidlin and Wentzell \cite{FW98}, and
has been further refined by Berglund and Gentz
using sample path techniques \cite{BeGe06} and by Bovier and his
collaborators using tools from potential theory, see Chapter 11
of \cite{BovH15}. We refer to the latter two monographs for further
information on the available literature.

Let us discuss already here the technical
differences between our approach and those cited above:
the method of potential theory relies on PDE connections and in
particular the generator of the process. The obstacle when trying
to apply these methods in our case is that one would have to deal with a
PDE on a domain that changes in time. The theory of partial
differential equations seems not to offer many useful quantitative
results in these cases. While in the regime of very slow pulling
the methods of \cite{BovH15} would presumably lead to at least
partial results, the
case of moderately slow pulling seems not tractable by these
methods. So instead, we rely on detailed large deviation results
for Gaussian processes, subdivision into suitable time
intervals, and
decoupling inequalities. This leads to sharp estimates on exit
time distributions and is somewhat related to the ideas
present in the approach pioneered by Berglund and Gentz
\cite{BeGe06}, see also \cite{BGK12, BGK15} for applications in
the context of fast-slow systems and mixed mode oscillations,
and \cite{BG14} for the passage through an unstable orbit.
Again, our model differs from those mentioned above by the presence of
the explicit time dependence, which prevents us from using large
parts of the established theory without significant loss of precision;
on the other hand, since we aim for less generality, our
results are sharper than most.
Finally, since we do not use semigroup theory or the generator at all,
many of our results are actually
valid for more general Gaussian processes than those derived from
\eqref{eqn:original}, see e.g.\ Lemma~\ref{lem:veryslowpullingstatlemma},
which may be of independent interest.

%
%
%
%
%

The model \eqref{eqn:original} has been treated before by Allman and Betz \cite{allman_betz_09}, and Allmann, Betz and Hairer
\cite{allman_betz_11}. In the present
paper, we significantly improve the results of \cite{allman_betz_09} in
several ways. First, we remove a logarithmic
gap between the regimes of fast and moderately slow pulling that
was present in \cite{allman_betz_09}. Second, we treat the regime
of very slow pulling, which was left out there. Third, and most
importantly, we are able to treat chains of arbitrary length, while
the results in \cite{allman_betz_09} were restricted to $d=2$,
i.e.\ only one mobile particle. The detailed results
on the exit times are also new.

So apart from the interesting transitional
regimes $\sigma \sim \eps$ and $\sigma^2 |\ln \eps| \sim 1$
where we do not know what happens, the model \eqref{eqn:original}
is now fairly well under control. However, the  simplifying assumption that the potential is quadratic
(or, more generally, convex), and the chain breaks at some
pre-determined bond length, is not harmless: if switching to a more
realistic potential (such as Lennard--Jones), the breaking behaviour of the
chain in dependence $\eps$ and $\sigma$ changes significantly,
see \cite{allman_betz_11} for the case $d=3$.
Nothing is known in these cases for chains of arbitrary length.

This paper is organized as follows. In Section \ref{sec:model}, we introduce the model
and the objects of our study. In Section~\ref{sec:mainresults}, we state the main
results, that is, we distinguish the three mentioned regimes and provide the limit
theorems for the break time and break position in
each regime.

Section~\ref{sec:prelim} is devoted to the key necessary technical tools; in particular,
we provide some useful representations of the solution to our model.
The proofs of the main results in the three regimes are given in Sections~\ref{sec:determ},
\ref{sec:nopulling}, and \ref{sec:moderate}, respectively.


\subsection{The model}
\label{sec:model}

Let $d\ge 1$ be an integer. We consider a chain of $d+1$ particles located
on the real line. At time $t\ge 0$, the locations of the particles are denoted by $X^0_t,\dots, X^d_t$,
and we assume that they satisfy the following system of stochastic differential equations:
\begin{equation} \label{eqn:original}
   \begin{cases}
     X_0^i = i & i=0,1,\ldots, d;\\
     X_t^0 = 0 & t\geq 0;\\
     X_t^d = d+\eps t & t\geq 0;\\
     \dd X_t^i = (X^{i+1}_t+X^{i-1}_t -2 X^i_t)\dd t+\sigma \dd B_t^i,& i=1,\ldots,d-1, t\geq 0,
   \end{cases}
\end{equation}
where $(B_t^i)_{t\geq 0}$ are independent Brownian motions, $i=1,\ldots,d-1$, $\sigma\ge 0$, and $\eps\ge 0$.

This means that one border particle (number $0$) is fixed at zero, another border particle (number $d$)
is moving at a constant speed $\eps$, and the neighboring particles interact according to the last (main)
equation.

We study the behaviour of the system when $\sigma\to 0$ and $\eps\to 0$. Our interest is focused on the
study of the time when the chain of interacting particles ``breaks'' and on understanding at which place
of the chain the break occurs.

Let us introduce the corresponding terminology and notation. We say that
the chain {\it breaks} at time $t$ if the maximal distance between the neighboring
particles attains $2$ for the first time at $t$. Formally, for $i\in\{1,\ldots,d\}$ let
$$ 
  \tau^i := \inf\{ t\geq 0 ~|~ X^i_t - X^{i-1}_t = 2 \}
$$ 
 and define the {\it break time} as
$$ 
  \tau:= \min_{1\le i \le d} \tau^i
    = \inf\{ t\geq 0 : \exists i\in\{1,\ldots,d\} ~:~  X^i_t - X^{i-1}_t = 2 \}.
$$ 

From the geometric point of view based on the observation of the process
${\bf X}_t:=(X^0_t,\dots,X^d_t)^\top$, the break time simply means the exit time
of ${\bf X}$ from a certain deterministic polytope.
Accordingly, we also call $\tau^i$, $\tau$, and other similar variables {\it exit times}.

We are also interested in the asymptotics of the distribution of the position where break occurs,
i.e.\ $(\P(\tau=\tau^i))_{i=1,\ldots,d}$.

A simple, but important observation is that there is a trivial deterministic bound for the break time
(unless $\eps=0$). Indeed, define $t_\ast$ by
$$ 
   t_\ast := \frac{d}{\eps}.
$$ 
Then for $t=t_\ast$ we have
\begin{eqnarray*} 
  \max_{i\in\{1,\ldots,d\}} \left(  X^i_t - X^{i-1}_t\right)
   &\ge&
  \frac{1}{d} \sum_{i=1}^{d} \left(  X^i_t - X^{if-1}_t\right)
\\ 
  &=& \frac{1}{d} \left(X^d_t-X^0_t\right)= \frac{d+\eps t}{d} = 2,
\end{eqnarray*}
and we see that $\tau\le t_\ast$, i.e. the exit (break) occurs
prior to $t_\ast$.

\subsection{Main results} \label{sec:mainresults}
We have to distinguish three different regimes: fast pulling $\sigma/\eps\to 0$ (or $\sigma=0$);
moderately slow pulling $\eps/\sigma\to 0$ and $\sigma^2 |\ln \eps| \to 0$;
and very slow pulling $\sigma^2 |\ln \eps| \to \infty$ (or $\eps=0$).

\paragraph{Fast pulling regime and deterministic setting ($\sigma=0$).}

Let us first consider the fast pulling regime, where we assume that $\sigma/\eps\to 0$ (or $\sigma=0$).
In this case, the limiting exit distribution is concentrated on the last exit position and the exit time
is deterministic up to lower order terms.

\begin{thm}  \label{thm:quasideterm}
Let $\eps\to 0, \sigma/\eps \to 0$ or let $\sigma=0$. Then the asymptotic probabilities of the exit
positions are described by the formula
\[
  \P(\tau=\tau^i) \to \begin{cases}
   1& i=d,
\\
   0& 1\le i\le d-1.
  \end{cases}
\]

Further,
$$
    \tau = t_\ast-\frac{(d-1)(2d-1)}{6} +o_P(1).
$$
\end{thm}

The proof of this theorem is given in Section~\ref{sec:determ}.

One can show that for each $i\in\{1,\ldots,d\}$
$$
  \tau^i = t_\ast + \frac{d^2-1}{6}-\frac{i(i-1)}{2} +o_P(1).
$$

\paragraph{Moderately slow pulling regime.}
This regime is characterized by the conditions
\begin{equation} \label{eqn:stretch}
    \sigma/\eps\to \infty \qquad \text{and} \qquad \sigma^2 |\ln \eps| \to 0.
\end{equation}
Define the parameters
\begin{equation}\label{eqn:parameterv0}
\begin{array}{l}
    v^2:=\frac{d-1}{2d},\qquad \gamma:=\sqrt{2} d v = \sqrt{d(d-1)},
\\
   A_1:=A_d:=\frac{d}{d-1},\qquad A_i:=\frac{2d}{d-1}, i\in\{2,\ldots,d-1\}.
\end{array}
\end{equation}
Recall that a random variable $\chi$ is {\it double exponential} (or Gumbel) with parameters $a,b>0$, if
$$
   \P( \chi \leq r ) = \exp( -a \exp( - b r)), \qquad r\in \R.
$$

The main result in the intermediate regime is as follows.

\begin{thm} \label{thm:intermediateslowpulling}
Assume that $\eqref{eqn:stretch}$ holds. Then, as $\eps,\sigma\to 0$,
\be \label{eqn:positions}
   \P(\tau=\tau^i) \to \begin{cases}
                           \frac{1}{d-1} & i\in\{2,\ldots,d-1\};\\
                           \frac{1}{2(d-1)} & i\in\{1,d\}.
                       \end{cases}
\ee
and we have the following weak limit theorems for the exit times:
$$ 
     \frac{\sqrt{\ln(\sigma/\eps)}}{\sigma/\eps}
        \left(t_\ast-\tau^i-\gamma \frac{\sigma}{\eps} \sqrt{\ln (\sigma/\eps)} \right)
     \tod \chi_i,\qquad i\in \{1,\ldots,d\},
$$ 
and
$$
     \frac{\sqrt{\ln (\sigma/\eps)}}{\sigma/\eps}
     \left(t_\ast-\tau-\gamma \frac{\sigma}{\eps} \sqrt{\ln (\sigma/\eps)}\right)
     \tod \chi_0,
$$
where $\chi_i$ is a double exponential random variable with parameters
$a_i:=vd A_i  / \sqrt{2\pi}$ for $i\in\{1,\ldots,d\}$,
$a_0:=\sum_{i=1}^d a_i = 2vd^2 /\sqrt{2\pi}$ and $b:=\sqrt{2}/(vd)$,
while $v$, $\gamma$, and the $A_i$ are defined in $\eqref{eqn:parameterv0}$.
\end{thm}

The theorem says that the asymptotic distribution of the exit position is uniform on the points $\{2,\ldots,d-1\}$,
while the points $\{1,d\}$ carry half the weight of the others. The asymptotic time scaling of the exit time is
best expressed as the remaining time until $t_\ast$: Namely, asymptotically $t_\ast-\tau^i$ is given by
a deterministic term $\gamma \frac{\sigma}{\eps} \sqrt{\ln (\sigma/\eps)}$ and a stochastic term,
$\frac{\sigma}{\eps} (\sqrt{\ln (\sigma/\eps)})^{-1} \chi_i$, with a double exponential random variable
$\chi_i$ with parameters not depending on $\sigma$ and $\eps$ anymore.

We remark that the condition \eqref{eqn:stretch} ensures that
$t_\ast  \gg \gamma \frac{\sigma}{\eps} \sqrt{\ln(\sigma/\eps)}$.

The proof of this theorem is given in Section~\ref{sec:moderate}.

\paragraph{Very slow pulling regime.}  

Let us finally consider the case of very slow pulling characterized by the condition
\begin{equation} \label{eqn:condveryslowpulling}
\sigma^2 |\ln \eps| \to \infty.
\end{equation}

The main result in this regime is as follows.

\begin{thm} \label{thm:veryslowpulling}
Assume that \eqref{eqn:condveryslowpulling} holds. Then, as $\eps,\sigma\to 0$,
the asymptotic distribution of exit positions is described by \eqref{eqn:positions},
while we have the following weak limit theorems for the exit times:
$$
   \tau^i \cdot (\sigma v)^{-1} \exp( - (\sigma v)^{-2}/2 )  \cdot \frac{A_i}{\sqrt{2\pi}}
   \tod \mathcal{E}, \qquad i\in \{1,\ldots,d\},
$$
and
$$
    \tau \cdot (\sigma v)^{-1} \exp( - (\sigma v)^{-2}/2 )  \cdot \frac{2d}{\sqrt{2\pi}}
    \tod \mathcal{E},
$$
where $\mathcal{E}$ is a standard exponential random variable and $v$ and the $A_i$ are
defined in \eqref{eqn:parameterv0}.
\end{thm}

We remark that the assertions of the last theorem also hold for $\eps=0$ and $\sigma\to 0$.

The asymptotic distribution of the exit position is the same as in the intermediate regime. However, the time scaling
of the exits is quite different and in fact much smaller. Indeed, $\tau^i$ is of order $\exp((\sigma v)^{-2}/2)$
(up to lower order terms), which is of smaller order than $t_\ast$ due to \eqref{eqn:condveryslowpulling}.

The proof of this theorem is given in Section~\ref{sec:nopulling}.

\paragraph{Ideas of the proofs and possible generalizations.}
Let us give a quick sketch of the idea of the proofs. Consider for simplicity the situation
in Theorem~\ref{thm:veryslowpulling}. Here, the exit happens on a time scale way before the pulling
significantly influences the dynamics. Therefore, we essentially deal with an exit of a process
with variance tending to zero out of a fixed deterministic set. One can imagine the mechanism as
follows: The process makes an attempt to exit the set, that is, each of the correlated components
tries to become large. Since the variance of the process is small ($\sigma\to 0$), the exit is a large
deviation event; and therefore the process typicially fails to exit. These attempts are repeated until
finally the process does exit. We show that the attempts are almost independent; for this purpose, one
manages to obtain an exponential decoupling (in time). Furthermore, one needs to show that the joint exit
of different components of the process is much more unlikely than the exit of a single component.
Finally, one needs to control the exit probability of one component, which requires suitably precise
large deviation results for Gaussian processes.

The proof of Lemma~\ref{lem:veryslowpullingstatlemma} below realizes the above-mentioned proof strategy,
for a stationary version of our system (and $\eps=0$), in the most representative way. This result
may be of independent interest due to the generality of its assumptions.

Let us further remark that, in the moderately slow and very slow pulling regimes, the initial condition
$X_0^i=i$, $i=0,\ldots, d$, has no influence on the results. One may very well start with different
initial conditions, as long as they do not depend on $\sigma$ and $\eps$ too much.
In fact, one step in the proofs is to show that the resulting  distributions of exit positions
and exit times are the same as if the process were started {\it stationary}.

There are various ways in which our results can be generalized. It is natural to look at more general interaction
potentials rather than the quadratic potential that gives the driving equation in (\ref{eqn:original}).
Even if the potential is quadratic, one can extend the interaction beyond the nearest neighbours. In this case,
the change of the potential translates into a change of the driving matrix (see (\ref{eqn:defnofA})); and our techniques are quite robust
with respect to the properties of that matrix (we only use that its eigenvalues are negative).

Another possibility is to look at the model (\ref{eqn:original}) driven by some other Gaussian processes
rather than Brownian motion. For example, the result might transfer to fractional Brownian motion.
The only place where the concrete form of the process (and in fact the Markov property) is used is the proof
of Lemma~\ref{lem:nonstr.4.2concrete}, where the exponential decoupling is proved.
We stress that apart from this place, we do not use the Markov property of the solution to (\ref{eqn:original}).

\section{Preliminaries}
\label{sec:prelim}
\subsection{Decomposition into deterministic and homogeneous part}

The first step of our studies is to show that the solution to \eqref{eqn:original}
can be rewritten in terms of a main deterministic part, an asymptotically negligible
deterministic part (${\bf \Delta}$), and a linearly transformed multi-dimensional
Ornstein--Uhlenbeck process.

\begin{lem} \label{lem:lem1} 
The following is the solution to $\eqref{eqn:original}$:
\begin{equation} \label{eqn:decompositionofx}
    X_t^i := \frac{i}{d} ( \eps t + d) + \eps \Delta^i_t + Y_t^i, \qquad i=0,\ldots,d,
\end{equation}
where $Y^0\equiv Y^d\equiv 0$, $\Delta^0\equiv\Delta^d\equiv 0$;
for $i=1,\ldots,d-1$, $\Delta^i$ are deterministic functions such that
 ${\bf \Delta}_t=(\Delta^1_t,\ldots,\Delta^{d-1}_t)^\top$ is the unique solution
of the system of ordinary differential equations:
\begin{equation} \label{eqn:determsystem}
   {\bf \Delta}' + {\bf \nu} = \MA {\bf \Delta},\qquad {\bf \Delta}(0)=0,
\end{equation}
where $\nu^i:=i/d$, $i=1,\ldots,d-1$, and ${\bf \nu}:=(\nu^1,\ldots,\nu^{d-1})$;
finally $(Y_t^i)$, $i=1,\ldots,d-1$, are the solutions
to the following system of stochastic differential equations
\begin{equation} \label{eqn:homogeneous}
    \begin{cases}
       Y_0^i = 0 & i=1,\ldots, d-1;\\
       \dd {\bf Y}_t = \MA {\bf Y}_t\dd t+\sigma \dd {\bf B}_t,& t\geq 0,
    \end{cases}
\end{equation}
where ${\bf Y}_t=(Y_t^1,\ldots,Y_t^{d-1})^\top$ and the $(d-1)\times(d-1)$-matrix $\MA$
is given by
\begin{equation} \label{eqn:defnofA}
   \MA:= \begin{pmatrix}
   -2 & 1 & 0 & 0& \ldots& 0 \\
    1 & -2 & 1 & 0 &\ldots& 0 \\
    0 & 1& -2 & 1 & &  \\
      &   & \ddots & \ddots & \ddots  \\
    0 &\ldots & & 1 & -2 & 1 \\
    0 &\ldots & & 0 & 1 & -2
\end{pmatrix}.
\end{equation}
\end{lem}

We note here that $\MA$ is symmetric and negative definite as the eigenvalues of $\MA$ are given by $\lambda_i := -2\left( 1 - \cos(i\pi/d)\right)$, $i=1,\ldots,d-1$,
see a general formula for tridiagonal Toeplitz matrices e.g.\ in \cite[formula (4)]{NPR}.
We will only use that $\lambda_i$ are all negative and that $(-\MA)^{-1/2}$ exits, but do not need the explicit form of the eigenvalues.

%
%
%
%
\begin{proof}[ of Lemma~\ref{lem:lem1}]
Let ${\bf Y}$ be a solution to \eqref{eqn:homogeneous} and let
${\bf \Delta}_t=(\Delta^1_t,\ldots,\Delta^{d-1}_t)^\top$ be the unique solution
to the system  \eqref{eqn:determsystem}.

Then define $X$ by equation \eqref{eqn:decompositionofx}. Clearly, $X$ satisfies
the correct initial conditions.
Furthermore, we obtain the stochastic differential equation for $X$:
\begin{eqnarray*}
  &&   [X^{i+1}_t +X^{i-1}_t-2X^{i}_t]\dd t + \sigma \dd B_t^i
\\
  &=&   \left[(\frac{i+1}{d}+\frac{i-1}{d}-2 \frac{i}{d}) (\eps t+d)
       + \eps (\Delta^{i+1}_t+\Delta^{i-1}_t-2\Delta^{i}_t)\right]\dd t
\\
  &&    +[Y^{i+1}_t +Y^{i-1}_t-2Y^{i}_t]\dd t + \sigma \dd B_t^i
\\
  &=&  \eps (\MA{\bf \Delta}_t)^i\dd t + (\MA {\bf Y}_t)^i\dd t + \sigma \dd B_t^i
\\
  &=&  \eps(\nu^i+  (\Delta^i)'_t)\dd t+ \dd Y^i_t
\\
  &=& \dd X^i_t,\qquad i=1,\ldots,d-1,
\end{eqnarray*}
as required, where we used \eqref{eqn:determsystem} and \eqref{eqn:homogeneous}
in the last but one step.
%
%
\end{proof}

The deterministic part ${\bf \Delta}$ from Lemma~\ref{lem:lem1} admits a further simplification.

\begin{lem} \label{lem:Delta_h_Z}
Let  the vector function ${\bf \Delta}$ be defined in $\eqref{eqn:determsystem}$.
Then it is true that
$$ 
  \Delta^i_t=h_i+Z^i_t, \qquad i\in\{1,\ldots,d-1\},
$$ 
where
\[
    h_i:= \frac{i(i^2-d^2)}{6d}, \qquad i\in\{0,\ldots,d\},
\]
and the functions $Z^i$ satisfy the conditions
\be \label{eqn:Zproperties}
   C^i_Z:=\sup_{t\ge 0} |Z^i_t|<\infty,\qquad \lim_{t\to \infty} Z^i_t=0,\qquad i\in\{1,\ldots,d-1\}.
\ee
\end{lem}

\begin{proof}
Consider the vector ${\bf h}:=(h_1,\dots h_{d-1})^\top$. It is easy to check that (using $h_0=h_d=0$)
$$
  \frac{i}{d} = h_{i+1}+h_{i-1}-2h_i, \qquad i\in\{1,\ldots,d-1\}.
$$
The latter relation means that $\MA{\bf h}={\bf \nu}$, where the matrix $\MA$
and the vector ${\bf \nu}$ were defined in Lemma~\ref{lem:lem1}.

Let  ${\bf Z}_t :={\bf \Delta}_t-{\bf h}$. Then the system \eqref{eqn:determsystem}
transforms into the {\it homogeneous} system
\[
   {\bf Z}' = \MA\, {\bf Z}, \qquad {\bf Z}(0)=-{\bf h}.
\]
Due to the fact that $\MA$ has only negative eigenvalues, it is clear that all components $Z^i$ of the solution
${\bf Z_t}=-e^{t\MA} {\bf h}$ are exponentially decreasing when the time goes to infinity,
thus they satisfy \eqref{eqn:Zproperties}.
\end{proof}

\subsubsection*{An example}
Consider the simplest non-trivial example. Let $d=2$; then we have two border particles moving
deterministically, $X^0_t\equiv 0$ and $X^2_t=2+\eps t$, and one middle particle performing
a nontrivial stochastic movement $X^1_t$. The negligible deterministic part satisfies the simple
ordinary differential equation
\[
   [\Delta^1]' + \frac12 = -2 \Delta^1,\qquad \Delta^1(0)=0,
\]
so that
\[
    \Delta^1_t = \frac{1}{4}\left(e^{-2t}-1\right).
\]
In the stochastic part, the matrix $\MA$ simplifies to the number $-2$, and we obtain
an Ornstein--Uhlenbeck process $Y^1_t=Y^{1,\sigma}_t$, i.e.\ satisfying equation
\[
   \dd Y^1_t = -2 Y^1_t \dd t + \sigma \dd  B_t,\qquad Y^1_0 = 0,
\]
where $B$ is a Brownian motion. By collecting all terms, we obtain
\[
   X^1_t = 1+ \frac{\eps t}{2} -\frac{\eps}{4}\left(1-e^{-2t}\right) + Y^{1,\sigma}_t.
\]

\subsection{Representation as a multi-dimensional Ornstein--Uhlenbeck processes}

Recall that a real-valued Gaussian process $U$ is called an \textit{Ornstein--Uhlenbeck process
started at zero} with parameter $\lambda<0$ if it solves the equation
$$
    \dd U_t = \lambda U_t \dd t + \dd  B_t,\qquad U_0 = 0,
$$
where $B$ is a Brownian motion.

A \textit{stationary Ornstein--Uhlenbeck process} is given by the solution to the stochastic differential equation
\begin{equation} \label{eqn:dglu0stat}
    \dd \tilde U_t = \lambda \tilde U_t \dd t + \dd  B_t,\qquad \tilde U_0 = (-2\lambda)^{-1/2} \xi,
\end{equation}
where $\xi$ is a standard normal random variable independent of the Brownian motion $(B_t)$. One can check that $(U_t)$ is indeed stationary.

Our next step consists in expressing the stochastic part of \eqref{eqn:decompositionofx},
i.e.\ ${\bf Y}$ being the solution to \eqref{eqn:homogeneous}, as a mixture of
independent real-valued Ornstein--Uhlenbeck processes {\it started at zero}. In the subsequent
evaluations, it will be convenient to replace these processes by their {\it stationary versions}.

\begin{lem} \label{lem:exprviaOU}
The following centered Gaussian process is a representation of the solution to $\eqref{eqn:homogeneous}$:
$$
    {\bf Y}_t := \sigma e^{\MA t} \int_0^t e^{-\MA u} \dd {\bf B}_u.
$$

Consider furthermore the stationary centered Gaussian process
$$
    {\bf \tilde Y}_t :=  \sigma e^{\MA t}( (-2\MA)^{-1/2} {\bf \xi} +  \int_0^t e^{-\MA u} \dd {\bf B}_u),
$$
with i.i.d.\ standard normal ${\bf \xi}:=(\xi_j)_{j=1,\ldots,d-1}$ independent of $({\bf B}_u)$. Then
\be \label{eqn:closenessU}
     ||{\bf Y}_t-{\bf \tilde Y}_t||_\infty \leq \sigma\,\sqrt{\frac{d-1}{2\mu}}\, e^{-\amu t} \max_{1\leq j\leq d-1} |\xi_j|,
\ee
where $\amu:=\min_{1\leq j \leq d-1} |\lambda_j|$, $\lambda_j$ are the eigenvalues of $\MA$.
\end{lem}

\begin{proof} One easily checks with the help of the It\^o formula the ${\bf Y}_t$ satisfies the SDE \eqref{eqn:homogeneous}. The explicit representation and the It\^o isometry also allows to show that ${\bf \tilde Y}_t$ is stationary. The difference ${\bf Y}_t-{\bf \tilde Y}_t$ can be estimated as follows: let $\MA=Q^\top D Q$ be a diagonalization of $\MA$. Then $(-2\MA)^{-1/2}=Q^\top (-2D)^{-1/2} Q$ and
\begin{eqnarray*}
||e^{\MA t}( -2\MA)^{-1/2} {\bf \xi}||_\infty  &=& || Q^\top e^{ D t } Q^T Q (-2D)^{-1/2} Q {\bf \xi}||_\infty
\\
 &\leq &  || Q^\top  e^{ D t } (-2D)^{-1/2} Q {\bf \xi}||_2
\\
& =&||  e^{ D t } (-2D)^{-1/2} Q {\bf \xi}||_2
\\
&\leq& \max_{1\leq j\leq d-1} e^{\lambda_j t} \cdot \max_{1\leq j\leq d-1} (-2\lambda_j)^{-1/2} \cdot || Q {\bf \xi} ||_2
\\
&= & e^{-\min_{1\leq j\leq d-1} |\lambda_j| t} \cdot (2\min_{1\leq j\leq d-1}|\lambda_j|)^{-1/2} \cdot || {\bf \xi} ||_2
\\
&\leq & e^{-\amu t} \cdot (2\amu)^{-1/2} \cdot \sqrt{d-1} \max_{1\leq j\leq d-1} |\xi_j|.
\end{eqnarray*}
\end{proof}

\subsection{Exit conditions for the solution and their stationary analogues}
The purpose of this subsection is to outline the sufficient and the necessary
exit conditions for the model in terms of an appropriately chosen {\it stationary}
process. These stationary versions of exit are easier to verify because for
stationary processes we have extremely robust bounds for large deviation probabilities
(see the Pickands lemma, Lemma~\ref{lem:pickands} below).

Define the $(d-1)\times d$ matrix $G$ by
$$
G_{i,j}:=\begin{cases} 1 & i=j,\\ -1 & i=j+1,\\ 0 & \text{otherwise}. \end{cases}
$$

Recall that
$$
    \tau^i = \inf\{ t\geq 0 ~|~ X^i_t - X^{i-1}_t = 2\}.
$$
Using the representations in Lemmas~\ref{lem:lem1} and~\ref{lem:exprviaOU}, we see that
\begin{eqnarray} \notag
    X^i_t - X^{i-1}_t &=& \frac{1}{d}(\eps t +d) + \eps (\Delta^i_t - \Delta^{i-1}_t)
        + Y^i_t - Y^{i-1}_t
\\  \notag
    &=& \frac{1}{d}(\eps t +d) + \eps (\Delta^i_t - \Delta^{i-1}_t)
        + (G {\bf Y}_t)^i,
\\  \label{eqn:differenceX}
    &=& \frac{1}{d}(\eps t +d) + \eps \delta_i + \eps (Z^i_t - Z^{i-1}_t) + \sigma V^i_t,
\end{eqnarray}
where ${\bf \Delta}$ and $Z$ are as in Lemma~\ref{lem:Delta_h_Z},
$\delta_i:=h_i-h_{i-1}=\tfrac{1-d^2}{6d}+ \tfrac{i^2-i}{2d}$ for $i\in\{1,\ldots,d\}$; while
$$ 
    {\bf V}_t:=(V^1_t,\ldots,V^{d-1}_t) := G e^{\MA t} \int_0^t e^{-\MA u} \dd {\bf B}_u =  \sigma^{-1} G {\bf Y}_t.
$$ 
Therefore, we can rewrite the exit time as follows
\be \label{eqn:tauiV}
     \tau^i = \inf\{ t\geq 0 \ | \  \sigma V_t^i
             =  1-\frac{\eps t}{d}  -  \eps (\delta_i+ Z^i_t - Z^{i-1}_t)  \}.
\ee
We will prefer to replace $V_t^i$ in this formula by its stationary version. Namely, let
\begin{equation} \label{eqn:Vti}
    {\bf \tV}_t:=(\tV^1_t,\ldots,\tV^{d-1}_t):=  G e^{\MA t}( (-2\MA)^{-1/2} {\bf \xi} +  \int_0^t e^{-\MA u} \dd {\bf B}_u) =  \sigma^{-1} G {\bf \tilde Y}_t,
\end{equation}
with the process $({\bf \tilde Y}_t)$ from Lemma~\ref{lem:exprviaOU} (in particular satisfying \eqref{eqn:closenessU}).
In Lemma~\ref{lem:7directv1} below, we shall analyze the local covariance structure of this
stationary Gaussian process.

It follows from \eqref{eqn:closenessU} that
\[
   | V_t^i - \tV_t^i | \leq \Xi\, e^{-\amu t},\qquad i=1,\ldots,d-1,
\]
for all $t\ge 0$, for $\amu:= \min_{1\le j\le d}|\lambda_j|$ and the finite positive random variable $\Xi:=2\sqrt{(d-1)/(2\amu)} \max_{j\in\{1,\ldots,d-1\}} |\xi_j|$.

Let us now consider the {\it exit event}
\be \label{eqn:exitevent}
    \left\{ \sigma V_t^i  \ge  1-\frac{\eps t}{d} -  \eps (\delta_i+ Z^i_t - Z^{i-1}_t) \right\}.
\ee
If \eqref{eqn:exitevent} occurs then we clearly have
\[
    \sigma \tV_t^i  \ge  1-\frac{\eps t}{d} -  \eps (\delta_i+ Z^i_t - Z^{i-1}_t) - \sigma \Xi\, e^{-\amu t}.
\]
Moreover, using \eqref{eqn:Zproperties} and letting $K^i:= |\delta_i|+ C_Z^i +C_Z^{i-1}$, 
we have
\be  \label{eqn:exitnecessaryshort}
   \sigma \tV_t^i  \ge  1-\frac{\eps t}{d}  - \eps K^i - \sigma \Xi\, e^{-\amu t}.
\ee
These two conditions in terms of $\tV^i$ are {\it necessary} for the exit. Similarly, the conditions
\[
    \sigma \tV_t^i  \ge  1-\frac{\eps t}{d} -  \eps (\delta_i+ Z^i_t - Z^{i-1}_t) + \sigma  \Xi\, e^{-\amu t}
\]
and
\be  \label{eqn:exitsufficientshort}
   \sigma \tV_t^i  \ge  1-\frac{\eps t}{d} + \eps K^i  +  \sigma \Xi\, e^{-\amu t},
\ee
are {\it sufficient} for the exit, i.e. if any of them is verified, then
\eqref{eqn:exitevent} occurs.

\subsection{Local covariance asymptotics}
The purpose of this subsection is to analyze the covariance expansion of the process $( {\bf \tV}_t)$.

%


\begin{lem} Consider the processes ${\bf \tV}_t=(\tV_t^1,\ldots,\tV_t^d)$ defined in $\eqref{eqn:Vti}$. Then, for all $i,j\in\{1,\ldots,d\}$, $\E [\tV_s^i \tV_t^j ]= \E [\tV_0^i \tV_{|t-s|}^j ]$
and as $t\to 0$, $i\in\{1,\ldots,d\}$,
\begin{equation} \label{eqn:lemma7directv1}
       \E[ \tV^i_0  \tV^i_t ] = v^2 ( 1 - A_i |t| + o(t)),
\end{equation}
where $v^2$ and the $A_i$ are defined in $\eqref{eqn:parameterv0}$. 
Further,
\begin{equation} \label{eqn:transfertov3lemma7}
\E[ \tV^i_0  \tV^j_0 ] = -\frac{1}{2d},\qquad i\neq j.
\end{equation}
\label{lem:7directv1}
\end{lem}

We remark that the family of covariances $\E [\tV_0^i \tV_t^j ]$, $i,j\in\{1,\ldots,d\}$, $t\in\R$,
can be expressed in terms of the family of matrices $\MA e^{\MA t}$. 
However, we will  not need these representations.

\medskip
\begin{proof}[ of Lemma~\ref{lem:7directv1}]  By stationarity, $\E [\tV_s^i \tV_t^j ]= \E [\tV_0^i \tV_{|t-s|}^j ]$.
Let us write $v\otimes w:=(v_i w_j)_{i,j=1,\ldots,d-1}$. One can check using the explicit form of $({\bf \tilde Y}_t)$ that
$$
\E[ {\bf \tilde Y}_0 \otimes {\bf \tilde Y}_t ] = \frac{\sigma^2}{2} (-\MA)^{-1} e^{t \MA}.
$$
This yields (using ${\bf \tilde V}=\sigma^{-1} G {\bf \tilde Y}$)
\begin{equation}\label{eqn:finexpr}
\E[ {\bf\tV}_0 \otimes {\bf \tV}_t ] = \frac{1}{2} G (-\MA)^{-1} e^{t \MA} G^\top
=\frac{1}{2} G (-\MA)^{-1} G^\top - G G^\top \, t + o(t),
\end{equation}
as $t\to 0$. A direct computation gives
$$ 
(G G^\top)_{i,j} = \begin{cases}
	1 & \text{if } i=j \in \{1,d\}, \\
	2 & \text{if } i=j \in \{2, \ldots, d-1\}, \\
	-1 & \text{if } |i-j|=1, \\
	0 & \text{otherwise},
\end{cases}
$$ 
which identifies the linear term in the expansion (\ref{eqn:finexpr}). Furthermore, we claim that
\begin{equation} \label{eqn:QainverseQast}
	G \mathcal (-\MA)^{-1}G^\top = I - \tfrac{1}{d} P_\eta,
\end{equation}
where $P_\eta$ is the projection onto the vector $\eta= (1, 1, \ldots, 1)^\top \in
\R^d$, i.e.\ $P_\eta w = (\eta,w) \eta$. Indeed, one readily computes that $G^\top G
= -\MA$, and thus
\[
G^\top (G (-\MA)^{-1} G^\top - I) = G^\top G (G^\top G)^{-1} G^\top -
G^\top = 0,
\]
which means that the range of $G (-\MA)^{-1} G^\top - I$
is contained in the null space of $G^\top$.
The latter is spanned by the vector $\eta$,
and it thus follows that
\[
G (-\MA)^{-1} G^\top = I + r P_\eta
\]
for some $r \in \R$. The constant $r$ is determined by applying the last equation to $\eta$ and using that $G^\top \eta = {\bf 0}$ so that ${\bf 0}=\eta + r (\eta,\eta) \eta = \eta + r d \eta$, and so \eqref{eqn:QainverseQast} follows.

Now, \eqref{eqn:QainverseQast} says that the first term in the expansion (\ref{eqn:finexpr}), $G (-\MA)^{-1} G^\top$, has diagonal elements $1 - 1/d=(d-1)/d$ and off-diagonal elements $-1/d$.
\end{proof}

\subsection{A decoupling inequality}
\label{sec:decoupling}

An important strategy in the proofs of our main results will be to decouple events that happen in disjoint and well-separated time intervals. For this purpose, we develop the necessary tools in this subsection.

Let $(W_t)_{t\in\R}$ be a stationary centered $d$-dimensional process
with finite second moments. Consider the associated linear spaces
\[
   \WW_t^- :=  \span( W_s^j, s\leq t, 1\le j\le d)
\]
and
\[
   \WW_t^+ :=  \span( W_s^j, s\geq t, 1\le j\le d).
\]
The quantity
$$
   r(\theta):=\sup\{ \covcoef(Z_0,Z): Z_0\in  \WW_0^-, Z\in \WW_\theta^+\}
$$
is called  {\it linear mixing coefficient}; and $\covcoef(X,Y)$ is the correlation coefficient of the random variables $X,Y$. On the other hand, the quantity
$$
   \mix(\theta):=\sup_{A \in \sigma(W_t, t\leq 0),B\in \sigma(W_t, t\geq \theta)}
       |\P(A\cap B) - \P(A)\cdot \P(B)|
$$
is called  {\it strong mixing coefficient} (or $\alpha$-mixing).
We refer to \cite{bradley2005} for a survey on notions of mixing and their relations.

As an important tool in the proofs, we shall use the following classical decoupling result
from the mixing theory.

\begin{lem} \label{lem:nonstr.4.2}
Let $(W_t)_{t\in\R}$ be a stationary centered $d$-dimensional Gaussian process.
Then  for every $\theta>0$ we have $\mix(\theta)\leq  r(\theta)$.
\end{lem}

The lemma tells us that the dependence between two events
that use information on $(W_t)$ for instants that are separated in time by at least $\theta$
can be evaluated via covariance characteristics.

\begin{proof}
It is true that $\mix(\theta) \leq \rho(\theta)= r(\theta)$. Here, $\rho$ is the {\it $\rho$-mixing coefficient} (see \cite{bradley2005}). Indeed,
the first inequality  is always true (see (1.12) in \cite{bradley2005}). The equality
between $\rho$ and $r$ is true because we deal with Gaussian processes, and it follows
from \cite[Theorem 1]{kolmogorovrozanov1960} (also see (7.1) in \cite{bradley2005}).
\end{proof}

\begin{lem} \label{lem:nonstr.4.2concrete}
The processes $(\tV_t^i)_{t\in\R}$ for $i\in\{1,\ldots,d\}$ defined in $\eqref{eqn:Vti}$ satisfy the relation
\be \label{eqn:mixingV}
   \sup_{A \in \sigma(\tV_t^i, t\leq 0, 1\le i\le d),B\in \sigma(\tV_t^i, t\geq \theta, 1\le i\le d))}
       |\P(A\cap B) - \P(A)\cdot \P(B)| \leq e^{-\amu \theta},
\ee
for all $\theta>0$ and $\amu:=\min_k|\lambda_{k}|$.
\end{lem}

\begin{proof}
Since $\tV_t^i\in \span(\tU_t^j, 1\le j\le d-1)$ where $(\tU_t^j)$ are independent Ornstein--Uhlenbeck processes with parameters $\lambda_j$, respectively, (cf.\ (\ref{eqn:dglu0stat})) by the definition (see (\ref{eqn:Vti})), it is enough to prove
\be \label{eqn:prefactors}
   \sup_{A \in \sigma(\tU_t^j, t\leq 0, 1\le j\le d-1), B\in \sigma(\tU_t^j, t\geq \theta, 1\le j\le d-1))}
       |\P(A\cap B) - \P(A)\cdot \P(B)| \leq e^{-\amu\theta}.
\ee
By Lemma~\ref{lem:nonstr.4.2}, we only need to check that for all $\theta>0$, and all random variables
$Z_0\in \span(\tU_s^j, s\le 0, j\le d-1)$, $Z\in \span(\tU_\tau^{j},\tau\ge \theta, j\le d-1)$,
 it is true that
\be \label{eqn:covarianceU}
   \covcoef(Z_0, Z) \le e^{- \amu\theta}.
\ee
We split the proof of \eqref{eqn:covarianceU} in few small steps.

{\it Step 1.}
Let $U_t$ be a standard real stationary Ornstein--Uhlenbeck process with covariance
$\E U_t U_s= e^{-|t-s|/2}$.

We fix $\theta\ge 0$ and observe that $\E[U_\theta| U_s, s\le 0]= e^{-\theta/2}U_0$
because
\[
  \E[(U_\theta-e^{-\theta/2}U_0)U_s] = e^{-(\theta-s)/2}- e^{-\theta/2}e^{s/2}=0, \qquad \forall s\le 0.
\]
Hence, for every $Z_0\in \span(U_s,s\le 0)$ we have
\be \label{eqn:covU_e1}
   \E[Z_0 U_\theta ]= \E [Z_0\, \E[U_\theta| U_s, s\le 0] ]
   = e^{-\theta/2} \E [Z_0 U_0]  \le  e^{-\theta/2} (\E Z_0^2)^{1/2}.
\ee

{\it Step 2.} Let now  $Z\in \span(U_\tau,\tau\ge \theta)$. By using the Markov property and the description of conditional expectations for two Gaussian random variables,
we get
\[
  \E[Z|U_s, s\le \theta] = \E[Z|U_\theta]= \E[Z U_\theta]\cdot  U_\theta.
\]

By the Cauchy-Schwarz inequality
\[
   \E[Z U_\theta] \le  (\E [Z^2])^{1/2} (\E [U_\theta^2])^{1/2} = (\E [Z^2])^{1/2}.
\]
Furthermore, for any $Z_0\in \span(U_s,s\le 0)$ we have, by using \eqref{eqn:covU_e1},
\begin{eqnarray} \nonumber
  \E[Z_0 Z]&=&\E[Z_0 \cdot \E[Z|U_\theta]]= \E[ Z_0 \cdot   \E[Z U_\theta] \cdot U_\theta ]
\\ \label{eqn:covU_e2}
  &=& \E[Z_0 \cdot U_\theta] \cdot \E[Z U_\theta]  \le   e^{-\theta/2} (\E [Z_0^2])^{1/2} (\E [Z^2])^{1/2}.
\end{eqnarray}

{\it Step 3.} Let now $\tU_t^j$, $1\le j\le d-1$, be independent centered stationary
Ornstein--Uhlenbeck processes with covariances $\E[\tU_s^j\tU_t^j]=e^{\lambda_j|s-t|}$ where $\lambda_j<0$.
(Compared to the processes used in the definition in $\eqref{eqn:Vti}$, we ignore the irrelevant constant
factors in front of the $\tilde U^j$ and let the variances be unit ones, which we can do w.l.o.g., as the
left-hand side of \eqref{eqn:prefactors} does not depend on constant prefactors of the $\tilde U^j$.)
Every such process $\tU_t^j$ can be reduced to a standard
OU-process by scaling of time. Therefore, \eqref{eqn:covU_e2} transforms into
\be \label{eqn:covU_e3}
    \E[Z_0 Z] \le  e^{\lambda_j \theta} (\E Z_0^2)^{1/2} (\E Z^2)^{1/2},
\ee
valid for all $Z_0\in \span(\tU_s^j,s\le 0)$ and $Z\in \span(\tU_\tau^j,\tau\ge \theta)$.

{\it Step 4.} Let us now fix an index $k$ and $Z\in \span(\tU_\tau^{k},\tau\ge \theta)$.
By using \eqref{eqn:covU_e3}, we have
\be \label{eqn:covU_e4}
  \E[\tU_0^{k} Z]\le e^{\lambda_{k} \theta} (\E [Z^2])^{1/2}.
\ee
By using the independence of the processes $\tU_t^j$ with different $j$, it is easy to see that
\be \label{eqn:covU_e5}
  \E[Z| \tU_s^j, s\le 0, 1\le j\le d-1] = \E[\tU_0^{k} Z] \tU_0^{k}.
\ee

{\it Step 5.} Let now  $Z\in \span(\tU_\tau^{j},\tau\ge \theta, 1\le j\le d-1)$.
Then $Z$ can be represented as a finite orthogonal sum
\[
   Z=\sum_{k=1}^{d-1} Z_k,  \qquad \textrm{with } Z_k\in  \span(\tU_\tau^{k},\tau\ge \theta),
\]
and by \eqref{eqn:covU_e5}  we have
\[
  \E[Z| \tU_s^j, s\le 0, j\le d-1] = \sum_{k=1}^{d-1} \E[Z_k|\tU_s^j, s\le 0, 1\le j\le d-1 ]
  =  \sum_{k=1}^{d-1} \E[\tU_0^{k} Z_k] \tU_0^{k}.
\]
Finally, by using Parseval's equation in the fourth step and \eqref{eqn:covU_e4} in the fifth step, we obtain for every
$Z_0\in \span(\tU_s^j, s\le 0, 1\le j\le d-1)$
\begin{eqnarray*}
    \E[Z_0 Z] &=& \E\left[ Z_0 \cdot \E[Z| \tU_s^j, s\le 0, 1\le j\le d-1] \right]
\\
   &=&     \sum_{k=1}^{d-1}  \E[\tU_0^{k} Z_k] \E [Z_0 \tU_0^{k}]
\\
  &\le&   \left(  \sum_{k=1}^{d-1}  \E[\tU_0^{k} Z_k]^2 \right)^{1/2} \cdot
          \left(  \sum_{k=1}^{d-1} [\E [Z_0 \tU_0^{k}]]^2 \right)^{1/2}
\\
   &\le&    \left(  \sum_{k=1}^{d-1}  \E[\tU_0^{k} Z_k]^2 \right)^{1/2} \E [Z_0^2]^{1/2}
\\
   &\le&   \left( \sum_{k=1}^{d-1} e^{2\lambda_{k} \theta}
           \E [Z_k^2] \right)^{1/2} \E [Z_0^2]^{1/2}
\\
  &\le&    e^{- \min_{1\le k\le d-1} |\lambda_{k}| \theta}
             \left(\sum_{k=1}^{d-1} \E[ Z_k^2] \right)^{1/2} \E [Z_0^2]^{1/2}
\\
   &=&  e^{-\amu\theta} (\E[Z^2])^{1/2}\E [Z_0^2]^{1/2},
\end{eqnarray*}
and \eqref{eqn:covarianceU} is verified.
\end{proof}

\subsection{Pickands lemma}
In the process of the proofs, we shall frequently use the following theorem that identifies
the exact asymptotics of  the large deviation probability of a stationary Gaussian process.
It is usually called Pickands lemma \cite{pickands}, \cite{michna}, also see Theorem~2.1
in \cite{QW} (with a slightly differently defined constant $\mathcal{H}_\alpha$)
and Theorem 9.15 in \cite{piterbarg}. The assertion with the varying time interval length
that we use here is due to V.\,Piterbarg.

\begin{lem} \label{lem:pickands}
Let $(W_t)_{t\in \R}$ be a stationary centered Gaussian process with covariance expansion
$$
    \cov(W_t,W_s)
    =v^2 \left( 1 - A |t|^\alpha + o(|t|^\alpha) \right),\qquad \text{as $t\to 0$},
$$
for some $v>0$, $A>0$, and $0<\alpha\leq 2$. Assume that
\[
  \limsup_{t\to\infty} \covcoef(W_t,W_0)<1.
\]
Then
$$
    \P( \max_{s\in[0,t]} W_s > x ) \sim
    \frac{A^{1/\alpha} \mathcal{H}_\alpha}{\sqrt{2\pi}} \cdot
       t \cdot (x/v)^{2/\alpha-1} \, \exp( -\frac{(x/v)^2}{2} ),
$$
for any $x$ and $t$ such that the right-hand side tends to zero,
where $\mathcal{H}_\alpha$ is Pickands constant (in particular, $\mathcal{H}_1=1$).
\end{lem}

\section{Deterministic regime ($\sigma=0$) and very fast pulling regime}
\label{sec:determ}

In this section, we show that in the deterministic regime (with $\eps\to 0, \sigma=0$)
the exit occurs at the last position, i.e. $\tau=\tau^d$. Then we extend this
result to stochastic systems
satisfying $\eps\to 0, \sigma/\eps \to 0$. It is therefore natural to call the
latter regime {\it quasi-deterministic}.

\subsection{Deterministic regime}

\begin{prop}
\label{prop:determ}
Let  $\eps\to 0, \sigma=0$. Then the exit times are described by the formula
\[
  \tau^i = t_\ast + \frac{d^2-1}{6}- \frac{i(i-1)}{2} +o(1),
  \quad \textrm{as } \eps\to 0,\quad 1\le i\le d.
\]
In particular,
\[
  \tau:=\min_{1\le i\le d} \tau^i = \tau^d = t_\ast -\frac{(d-1)(2d-1)}{6} +o(1),
  \, \quad \textrm{as } \eps\to 0.
\]
\end{prop}

 \begin{proof}[ of Proposition~\ref{prop:determ}]
Observe that in the deterministic regime the stochastic term $\sigma V_t^i$ vanishes
from \eqref{eqn:tauiV} and we have
\[
  \tau^i = \inf\{ t\geq 0 \ |
               \frac{t_\ast - t}{t_\ast}  =  \eps (\delta_i+ Z^i_t - Z^{i-1}_t) \}
\]
where $t_\ast=\tfrac{d}{\eps}$ and  $\delta_i:=\tfrac{1-d^2}{6d}+\tfrac{i(i-1)}{2d}$.
Using \eqref{eqn:Zproperties} we obtain
\begin{eqnarray*}
  \tau^i&=& t_\ast - d (\delta_i+ Z^i_{\tau^i} - Z^{i-1}_{\tau^i})
\\
  &=& t_\ast - d (\delta_i+o(1))
\\
  &=& t_\ast + \tfrac{d^2-1}{6}-\tfrac{i(i-1)}{2} +o(1),
\end{eqnarray*}
which is the claim of the proposition.
\end{proof}

\subsection{Quasi-deterministic regime}

We now extend the results for the deterministic regime to the very fast
pulling (quasi-deterministic) regime.

\begin{proof}[ of Theorem~\ref{thm:quasideterm}]
For identifying  the end of the chain as the exit position, it is sufficient
to prove that for each $i$ such that $1\leq i<d$ one has
\be \label{eqn:tauid}
   \P( \tau=\tau^i\le \tau^d)\to 0.
\ee
Let us fix an $i<d$ and some $M>0$.
The following inequality is the starting point for proving \eqref{eqn:tauid},

\be \label{eqn:tauid0}
    \P(\tau=\tau^i\le \tau^d)\le \P(\tau^i\le t_\ast-M)
           + \P( t_\ast-M \le \tau^i\le t_\ast, \tau^i\le \tau^d).
\ee
{\it Step 1:} Let us proceed with the evaluation of the second probability in
\eqref{eqn:tauid0}.
Assuming that $\tau^i\le \tau^d$, letting $t:=\tau^i$,
and using expression \eqref{eqn:differenceX} we transform the inequality
\[
    X^i_t-X^{i-1}_t = 2 \ge  X^d_t-X^{d-1}_t
\]
into
\[
   \sigma (V^i_t-V^d_t)
   \ge
   \eps (\delta_d-\delta_i -Z^i_t+Z^{i-1}_t + Z^d_t-Z^{d-1}_t),
\]
at $t=\tau^i$. Recall that $t\ge t_\ast-M\to\infty$; by \eqref{eqn:Zproperties}
we have
\[
    Z^i_t-Z^{i-1}_t - Z^d_t+Z^{d-1}_t =o(1).
\]
We have thus seen that on $\tau^i\leq\tau^d$
\[
  \sup_{s\in [t_\ast-M,t_\ast]} (V^i_s-V^d_s) \ge (V^i_t-V^d_t)
  \ge \frac{\eps}{\sigma} \, (\delta_d-\delta_i +o(1)),
\]
and recall that $\delta_d>\delta_i$ for $i<d$.

Therefore, we obtain
\begin{eqnarray*}
    &&  \P( t_\ast-M \le \tau^i\le t_\ast, \tau^i\le \tau^d)
\\
   &\le& \P\left(
      \sup_{s\in [t_\ast-M,t_\ast]} (V^i_s-V^d_s)
      \ge \frac{\eps}{\sigma} \, (\delta_d-\delta_i +o(1))
         \right)
\\
    &\le& \P\left(
      \sup_{s\in [t_\ast-M,t_\ast]} |V^i_s-V^d_s|
      \ge \frac{\eps}{\sigma} \, (\delta_d-\delta_i +o(1))
            \right)
\\
    &\le& \P\left(
      \sup_{s\in [t_\ast-M,t_\ast]} |\widetilde{V}^i_s-\widetilde{V}^d_s|
      \ge \frac{\eps}{\sigma} \, (\delta_d-\delta_i +o(1))
            \right)
\\
    &=& \P\left(
      \sup_{s\in [0,M]} |\widetilde{V}^i_s-\widetilde{V}^d_s|
      \ge \frac{\eps}{\sigma} \, (\delta_d-\delta_i +o(1))
           \right) \to 0
\end{eqnarray*}
for every fixed $M$.
Here we used a stationary version $\widetilde{V}^i$ of $V^i$, applied the Anderson inequality
(3rd step), then stationarity (4th step), and finally used that
$\frac{\eps}{\sigma}\to\infty$.
\medskip

{\it Step 2:}  Let us now evaluate the first probability in \eqref{eqn:tauid0},
assuming additionally that $M$ is chosen so large that
$\frac{M}{2d} >  K^i:=|\delta_i| + C^i_Z+C^{i-1}_Z$, where the constant $C_Z^i$ is defined in
\eqref{eqn:Zproperties}.

Using the representation \eqref{eqn:differenceX} for $t=\tau^i$ we see that
the exit condition $ X_t^i - X_t^{i-1}=2$  is equivalent to
\[
   \sigma V^i_t = \frac{t_\ast-t}{t_\ast}- \eps ( \delta_i + Z_t^i-Z_t^{i-1})
     =\eps \left( \frac{t_\ast-t}{d}- (\delta_i + Z_t^i - Z_t^{i-1}) \right),
\]
hence,
\[
    V^i_t = \frac{\eps}{\sigma}
            \left( \frac{t_\ast-t}{d}- (\delta_i + Z_t^i-Z_t^{i-1}) \right)
   \ge \frac{\eps}{\sigma} \left( \frac{t_\ast-t}{d}- K^i \right).
\]
Notice that the right hand side is positive for all $t\le t_\ast-M$ due
to the choice of $M$. Therefore,
\begin{eqnarray*}
  \P( \tau^i\le t_\ast-M)
   &\le&
  \P\left(\exists t\in[0,t_\ast-M]\ :\
  V^i_t \ge \frac{\eps}{\sigma} \left( \frac{t_\ast-t}{d}- K^i \right)\right)
\\
   &\le&
 \P\left(\exists t\in[0,t_\ast-M]\ :\
 |V^i_t| \ge \frac{\eps}{\sigma} \left( \frac{t_\ast-t}{d}- K^i\right)\right)
\\
   &\le&
  \P\left(\exists t\in[0,t_\ast-M]\ :\ |\tV^i_t|
     \ge \frac{\eps}{\sigma} \left( \frac{t_\ast-t}{d}- K^i\right)\right).
\end{eqnarray*}
Here we used again the stationary version  $\tV^i$ of  $V^i$
and applied the Anderson inequality.
Furthermore, by stationarity, and using $\frac{M}{2d} \ge K^i$, we have
\begin{eqnarray*}
 &&\P\left(\exists t\in[0,t_\ast-M]\ :\
         |\tV^i_t| \ge \frac{\eps}{\sigma}
         \left( \frac{t_\ast-t}{d}- K^i\right)\right)
 \\
   &\le&
  \P\left(\exists s\in[M,\infty)\ :\
        |\tV^i_s| \ge \frac{\eps}{\sigma}
        \left( \frac{s}{d}-  K^i\right)\right)
\\
   &\le&
       \P\left(\exists s\in[M,\infty)\ :\
       |\tV^i_s| \ge \frac{\eps}{\sigma} \, \frac{s}{2d} \right)
   \to 0,
\end{eqnarray*}
since $\frac{\eps}{\sigma} \to\infty$. We conclude that
\[
   \P( \tau^i\le t_\ast-M) \to 0,
      \qquad \textrm{as }  \frac{\eps}{\sigma} \to\infty \ \textrm{and }
      \frac{M}{2d}\ge K^i.
\]
Now \eqref{eqn:tauid0} yields the result for the exit position.

{\it Step 3:} We finally prove the statement about the exit time.
Recall that for large enough $M$ we already proved
$\P(\tau^d<t_\ast-M)\to 0$.

Fix a small $\kappa>0$. Recall that by \eqref{eqn:Zproperties} the term $Z_t^d-Z_t^{d-1}$ tends to zero,
as $t\to\infty$,
so that for $t>t_\ast-M$ we have, say,
\[
   |Z_t^d-Z_t^{d-1}| \leq \frac{\kappa}{2d}.
\]
Therefore, we obtain from  \eqref{eqn:differenceX}  
\begin{eqnarray*}
    &&\P(t_\ast-M \leq \tau^d \leq t_\ast - d \delta_d - \kappa)
\\
   &\le& \P( \exists t \in [t_\ast-M,t_\ast - d \delta_d - \kappa] :
            V_t^d \geq \frac{\eps}{\sigma} ( \frac{t_\ast - t}{d}
            - \delta_d - (Z_t^d - Z_t^{d-1}) ) )
\\
   &\leq & \P\left( \exists t \in [t_\ast-M,t_\ast - d \delta_d - \kappa] :
            V_t^d \geq \frac{\eps}{\sigma} ( \kappa/d - \kappa/(2d) ) \right)
\\
    &\leq & \P\left( \exists t \in [t_\ast-M,t_\ast - d \delta_d - \kappa] :
    V_t^d \geq \frac{\eps}{\sigma}\cdot \frac{\kappa}{2d} \right) \to 0,
\end{eqnarray*}
as above. Since $\kappa$ can be chosen arbitrarily small, this estimate shows
that $\tau^d \geq t_\ast - d \delta_d - o_P(1)$.

On the other hand, for any fixed small $\kappa>0$ let $t=t(\eps,\kappa):= t_\ast - d \delta_d +\kappa$.
Then, by using again \eqref{eqn:differenceX} with $i=d$,
\begin{eqnarray*}
    \P( \tau^d \leq t)
    &\geq& \P \left( V_t^d \geq \frac{\eps}{\sigma}
    \big( \frac{t_\ast-t}{d} - \delta_d -(Z_t^d - Z_t^{d-1}) \big) \right)
\\
    &\geq&\P\left( V_t^d \geq \frac{\eps}{\sigma}
           \big( \frac{t_\ast-t}{d} - \delta_d +\frac{\kappa}{2d} \big)\right)
\\
    &=&\P\left( V_t^d \geq  - \frac{\eps}{\sigma} \cdot \frac{\kappa}{2d} \right)
     \to 1,
\end{eqnarray*}
because the law of $V_t^d$ converges to a non-degenerated Gaussian law, and the level
tends to $-\infty$.
Since $\kappa$ can be chosen arbitrarily small, this estimate shows that
$\tau^d \leq t_\ast-d \delta_d +o_P(1)$.
\end{proof}

\section{The case of very slow pulling and no pulling}
\label{sec:nopulling}
\subsection{Result for the stationary model}
We start with considering a stationary model that captures the main features of our
(nonstationary) problem. We will make later a relatively easy passage from the
stationary case to the non-stationary one.

Let $\tV_t^i$ be the stationary Gaussian processes defined in \eqref{eqn:Vti}.
Define the related exit times
\[
  \ttau^i:=\inf\left\{ t: \sigma \tV_t^i \ge 1  \right\}
\]
and $\ttau:=\min_{1\le i\le d} \ttau^i$ (cf.\  \eqref{eqn:differenceX}). The main result for stationary case is as follows.

\begin{prop} \label{prop:veryslowpullingstat}
Let $\sigma\to 0$. Then
$$
\P( \ttau=\ttau^i ) \to \begin{cases}
                         \frac{1}{d-1} & i\in\{2,\ldots,d-1\};\\
                         \frac{1}{2(d-1)} & i\in\{1,d\}.
                      \end{cases}
$$
Further, for any $i\in\{1,\ldots,d\}$,
$$
\ttau^i \cdot (\sigma v)^{-1} \exp( - (\sigma v)^{-2}/2 ) \cdot \frac{A_i}{\sqrt{2\pi}}\tod \mathcal{E},
$$
and
$$
\ttau \cdot (\sigma v)^{-1} \exp( - (\sigma v)^{-2}/2 ) \cdot \frac{2d}{\sqrt{2\pi}}\tod \mathcal{E},
$$
where $\mathcal{E}$ is a standard exponential random variable, $v$ and the $A_i$ are defined in \eqref{eqn:parameterv0}.
\end{prop}

\begin{rem}\label{rem:sigmastar}
At some point we will need a minor extension of Proposition~$\ref{prop:veryslowpullingstat}$
allowing a very mild flexibility of the exit times. Namely, let $\sigma_{i*}$, $1\le i\le d$,
be so close to $\sigma$ that $\sigma_{i*}^{-2}=\sigma^{-2}+o(1)$. Let
\[
  \ttau_*^i:=\inf\left\{ t: \sigma_{i*} \tV_t^i \ge 1  \right\}
\]
and $\ttau_*:=\min_{1\le i\le d} \ttau_*^i$. Then we still have, as $\sigma\to 0$,
\begin{eqnarray*}
&& \P( \ttau_*=\ttau_*^i ) \to
                      \begin{cases}
                         \frac{1}{d-1} & i\in\{2,\ldots,d-1\};\\
                         \frac{1}{2(d-1)} & i\in\{1,d\}.
                      \end{cases}
\\
    && \ttau_*^i \cdot (\sigma v)^{-1} \exp( - (\sigma v)^{-2}/2 )
       \cdot \frac{A_i}{\sqrt{2\pi}}\tod \mathcal{E},
\\
   &&  \ttau_* \cdot (\sigma v)^{-1} \exp( - (\sigma v)^{-2}/2 )
       \cdot \frac{2d}{\sqrt{2\pi}}\tod \mathcal{E}.
\end{eqnarray*}
\end{rem}

Proposition~\ref{prop:veryslowpullingstat} will follow from the next lemma,
which does not rely on the explicit form of the processes $\tV_t^i$ but captures
their important features.

\begin{lem} \label{lem:veryslowpullingstatlemma}
Let $(\VVV_t^1,\ldots,\VVV_t^d)_{t\in\R}$ be a family of centered stationary
Gaussian process with the following properties:
\begin{enumerate}
\item[(V1)] The variance of all processes is the same:
\begin{equation}\label{eqn:localvar}
    \E[ \VVV_t^i ] = v^2>0,\qquad i=1,\ldots,d.
\end{equation}
\item[(V2)] The local covariance behavior is as follows: there is an $\alpha\in(0,2]$
such that for any $i$ there is $A_i>0$ with
\begin{equation}\label{eqn:localcov}
     \cov(\VVV_t^i,\VVV_0^i)=v^2 (  1 - A_i |t|^{\alpha} + o(|t|^\alpha) ),
     \qquad \text{as $t\to 0$}.
\end{equation}
\item[(V3)] There is a non-degeneracy among the components:
$$
    \sup_{s,t} \sup_{i,j\in\{1,\ldots,d \}, i\neq j} \covcoef ( \VVV_t^i,\VVV_s^j ) < 1. 
$$
\item[(V4)] There is a mixing property: i.e.\ there is a number $\mu>0$
such that
$$
   r(\theta) \leq e^{-\mu \theta},\qquad \theta>0,
$$
where
$$
     r(\theta) := \sup\left\{ \covcoef (W,W') :
     \genfrac {}{}{0pt}{}
      {W\in \span( \VVV_t^i, t\leq 0, i\in\{1,\ldots,d\})}
      { W'\in \span(\VVV_t^i, t\geq \theta, i\in\{1,\ldots,d\})}
      \right\}.
$$
\end{enumerate}
Then, for $\tau^i := \min\{ t>0 : \sigma \VVV_t^i \geq 1\}$ and $\tau:=\min_{i\in\{1,\ldots,d\}} \tau^i$
we have, as $\sigma\to 0$,
\begin{equation}\label{eqn:vspdistrib1}
     \P( \tau=\tau^i ) \to \frac{A_i^{1/\alpha}}{\sum_{j=1}^d A_j^{1/\alpha}},
\end{equation}
where $\alpha$ and the $A_i$ are as in \eqref{eqn:localcov}.

Further, as $\sigma\to 0$, we have the weak limit theorems
\begin{equation}\label{eqn:taulimitepszero}
     \tau^i \cdot (\sigma v)^{1-2/\alpha} \exp( - (\sigma v)^{-2}/2 )
     \cdot \frac{A_i^{1/\alpha} \mathcal{H}_\alpha }{\sqrt{2\pi}}\tod \mathcal{E},
\end{equation}
for $i\in\{1,\ldots,d\}$ and
\begin{equation}\label{eqn:taulimitepszero-taus}
   \tau \cdot (\sigma v)^{1-2/\alpha} \exp( - (\sigma v)^{-2}/2 )
   \cdot \frac{\sum_{j=1}^d A_j^{1/\alpha} \mathcal{H}_\alpha}{\sqrt{2\pi}} \tod \mathcal{E},
\end{equation}
where $\mathcal{E}$ is a standard exponential random variable and $\mathcal{H}_\alpha$ is Pickands constant (cf.\ Lemma~\ref{lem:pickands}).
\end{lem}

\begin{rem} A similar proof works if the covariance expansion at zero \eqref{eqn:localcov} is given
by $v^2(1-A_i |t|^{\alpha_i} +...)$ with some possibly distinct $\alpha_1,\ldots,\alpha_d\in(0,2]$.
The exit distribution is then concentrated on the positions $i$ where
$\alpha_i= \min_{j\in\{1,\ldots,d\}} \alpha_j$. The same applies if the variances of the different
components in \eqref{eqn:localvar} are different, say $=:v_i^2$. Then the exit distribution is
concentrated on those positions $i$ where $v_i^2=\max_{j\in\{1,\ldots,d\}} v_j^2$.
\end{rem}

\subsection{Proof of Lemma~\ref{lem:veryslowpullingstatlemma}}
{\it Step 1: Preliminaries.}

We start by remarking that in order to prove \eqref{eqn:vspdistrib1} it is sufficient to show
that, for all $i\in\{1,\ldots,d\}$,
\begin{equation}\label{eqn:vspdistrib1upper}
      \limsup_{\sigma\to 0} \P( \tau=\tau^i ) \leq \frac{A_i^{1/\alpha}}{\sum_{j=1}^d A_j^{1/\alpha}},
\end{equation}
as the right-hand side sums up to one.

Fix $i$ for the rest of this proof and set
$$
    T_\ast:=\sigma^{2/\alpha-1} \exp( (\sigma v)^{-2}/2 ),\qquad \ell:=B \sigma^{-2},\qquad
    L:=\sigma^{-3},\qquad M:=\frac{\mathcal{M} T_\ast}{L+\ell},
$$
where $B>0$ is a large constant chosen later and $\mathcal{M}>0$ is fixed (which we will send
to infinity later).

Consider the intervals
$$
    I_m:=[(L+\ell)(m-1),(L+\ell)(m-1)+L],\quad J_m:=[(L+\ell)(m-1)+L,(L+\ell)m],
$$
for $m\in\{1,2,\ldots,M\}$; and note that the disjoint union of these intervals is the
interval $[0,\mathcal{M} T_\ast]$.

For $m\in\{1,2,\ldots,M\}$, define the events
$$
    E_m^i := \{ \max_{t\in I_m} \VVV_t^i \geq \sigma^{-1} \},\qquad
    N_m:=\{ \max_{j\in\{1,\ldots,d\}} \max_{t\in I_m} \VVV_t^j < \sigma^{-1} \}.
$$

The event $E_m^i$ means that the $i$-th component produces an exit (in the sense of the statement of
the lemma) in the time interval $I_m$, while $N_m$ means that none of the components exits during the
time interval $I_m$. By the stationarity of the processes $(\VVV_t^i)_{t\in\R}$, $i\in\{1,\ldots,d\}$,
$$
    \P( E_1^i ) = \P( E_2^i ) = \ldots = \P( E_M^i ),\qquad \P( N_1 ) = \P( N_2 )
    = \ldots = \P( N_M ).
$$

The basic relation in order to prove \eqref{eqn:vspdistrib1upper} is the observation
\begin{equation}\label{eqn:thirtyone}
     \P( \tau=\tau^i ) \leq \P( \tau^i \geq \mathcal{M} T_\ast ) + \P( \tau^i \in \bigcup_{m=1}^M J_m )
     + \sum_{m=1}^M \P( \tau=\tau^i\in I_m).
\end{equation}

The first probability will tend to zero (first $\sigma\to 0$ then $\mathcal{M}\to\infty$), as an exit
far beyond the critical time scale $T_\ast$ is unlikely. The second probability will also tend to zero,
because the intervals $J_m$ are too short to produce an exit at all. The only contribution comes from the
last term.

{\it Step 2: The second term in \eqref{eqn:thirtyone} tends to zero.}

Observe that
$$
   \P( \tau^i \in \bigcup_{m=1}^M J_m ) \leq  \sum_{m=1}^M \P( \tau^i \in J_m )
   \leq  M \P( \max_{t\in J_1} \VVV_t^i \geq \sigma^{-1} ).
$$
Using stationarity again and the Pickands lemma (Lemma~\ref{lem:pickands}) with (V2),
we can estimate the last term as follows:
\begin{eqnarray*}
    &&M \P( \max_{t\in J_1} \VVV_t^i \geq \sigma^{-1} ) \notag
\\
    &=&  M \P( \max_{t\in [0,\ell]} \VVV_t^i \geq \sigma^{-1} ) \notag
\\
     &\leq &  M \cdot 2 \, \frac{A_i^{1/\alpha} \mathcal{H}_{\alpha}}{\sqrt{2\pi}}
              \cdot \ell \cdot (\sigma v)^{1-2/\alpha} \, \exp( -(\sigma v)^{-2}/2 ) \notag
\\
     &= & C_{\alpha,A_i,v} \frac{\mathcal{M} T_\ast}{L+\ell}\, \ell \cdot
     \sigma^{1-2/\alpha} \exp( -(\sigma v)^{-2}/2 ) \notag
\\
     &= & C_{\alpha,A_i,v,\mathcal{M}} \frac{\sigma^{2/\alpha-1}
     e^{(\sigma v)^{-2}/2}}{L+\ell}\, \ell \sigma^{1-2/\alpha}  \exp( -(\sigma v)^{-2}/2 )
     \notag
\\
     &= & C_{\alpha,A_i,v,\mathcal{M}} \frac{\ell}{L+\ell} \to 0, 
\end{eqnarray*}
as $\sigma\to 0$. We have thus seen that the second term in \eqref{eqn:thirtyone}
 tends to zero.

{\it Step 3: The first term in \eqref{eqn:thirtyone} tends to zero, as first $\sigma\to 0$
and then $\mathcal{M}\to\infty$.}

First note -- from Assumption (V2) -- that the Pickands lemma (Lemma~\ref{lem:pickands}) gives us
that
\begin{equation}
\label{eqn:pickands4.1}
     \P(E_1^i)=\P( \max_{ t\in[0, L]} \VVV^i_t \geq \sigma^{-1} ) \sim \frac{A_i^{1/\alpha}
     \mathcal{H}_{\alpha}}{\sqrt{2\pi}}\cdot L \cdot(\sigma v)^{1-2/\alpha}
     \cdot \exp( - (\sigma v)^{-2}/2 ),
\end{equation}
because the right-hand side indeed tends to zero, as the exponential term tends to zero
faster than the increasing polynomial term $L \sigma^{1-2/\alpha}$.

Further, we have to decouple the events $E_m$.
For this purpose, we use Assumption (V4), which makes Lemma~\ref{lem:nonstr.4.2} applicable:
Namely, by applying iteratively Lemma~\ref{lem:nonstr.4.2} using that
$E_m^i\in\sigma(\VVV_t^i, t\in I_m)$ and the fact that the intervals $I_m$ are separated
by a distance of at least $\ell$, we obtain
\begin{eqnarray}
    \P( \tau^i \geq \mathcal{M} T_\ast ) \notag
    &\leq & \P( \bigcap_{m=1}^M (E_m^i)^c ) \notag
\\
     &\leq & \P((E_1^i)^c) \cdot \P( \bigcap_{m=2}^M (E_m^i)^c ) + e^{-\amu\ell} \notag
\\
     &\leq & \P((E_1^i)^c) \cdot\P((E_2^i)^c)\cdot \P( \bigcap_{m=3}^M (E_m^i)^c )
     +\P((E_1^i)^c) \cdot e^{-\amu\ell} + e^{-\amu\ell} \notag
\\
     &\leq & \ldots \notag
\\
     &\leq & \prod_{m=1}^M \P((E_m^i)^c) + M e^{-\amu\ell} \notag
\\
     &= & \P((E_1^i)^c)^M + M e^{-\amu\ell} \notag
\\
     &= & (1-\P(E_1^i))^M + M e^{-\amu\ell} \notag
\\
     &\leq & \exp(-M\P(E_1^i)) + M e^{-\amu\ell}.  \label{eqn:copyforlimittheoremD}
\end{eqnarray}
We will show that both terms tend to zero, as first $\sigma\to 0$ and then
$\mathcal{M}\to\infty$.

Let us start with the second term. Here,
\begin{equation} \label{eqn:useforlowerboundtaus}
    M e^{-\amu\ell} = \frac{\mathcal{M} T_\ast}{L+\ell} \, e^{-\amu B \sigma^{-2} }
    = \frac{\mathcal{M} \sigma^{2/\alpha-1} \exp( (\sigma v)^{-2} /2 ) }{L+\ell} \,
    e^{-\amu B \sigma^{-2} },
\end{equation}
which tends to zero if $B$ is chosen sufficiently large ($\amu B > v^{-2}/2$; in fact,
later we will choose $B$ such that even $\amu B>v^{-2}$).

For the first term in \eqref{eqn:copyforlimittheoremD}, observe that by \eqref{eqn:pickands4.1}
 -- using the abbreviation
$c_{\alpha,A_i,v}:=A_i^{1/\alpha} \mathcal{H}_\alpha v^{1-2/\alpha} /\sqrt{2\pi}$ --,
we have that for $\sigma\to 0$
\begin{eqnarray}
   M\cdot \P(E_1^i) &\sim& \frac{\mathcal{M} T_\ast}{L+\ell}
   \cdot c_{\alpha,A_i,v} L \sigma^{1-2/\alpha} e^{-(\sigma v)^{-2}/2} \notag
\\
   &=& \frac{\mathcal{M} \sigma^{2/\alpha-1} e^{(\sigma v)^{-2}/2}}{L+\ell}
   \cdot c_{\alpha,A_i,v} L \sigma^{1-2/\alpha} e^{-(\sigma v)^{-2}/2} \notag
\\
   &\sim&  \mathcal{M} \cdot c_{\alpha,A_i,v}. \label{eqn:saveforlimt59}
\end{eqnarray}

We have thus seen that for any $\mathcal{M}>0$ we have
\begin{equation} \label{eqn:copyforlimittheorem}
      \limsup_{\sigma\to 0} \P( \tau^i \geq \mathcal{M} T_\ast )
      \leq e^{-\mathcal{M} c_{\alpha,A_i,v} }.
\end{equation}

Now, letting $\mathcal{M}\to \infty$ shows that the first term in \eqref{eqn:thirtyone}
tends to zero (first $\sigma\to 0$, then $\mathcal{M}\to\infty$).

{\it Step 4: The last term in \eqref{eqn:thirtyone} gives \eqref{eqn:vspdistrib1upper}.}

We shall use again Lemma~\ref{lem:nonstr.4.2}. Namely, observe that since the intervals
$I_k$, $k\in\{1,\ldots,m\}$ are separated by a distance of at least $\ell$, and that
$E_m^i\in\sigma(\VVV_t^i, t\in I_m)$ and similarly for the $N_m$. Therefore,
we obtain for $m\in\{1,\ldots,M\}$
\begin{eqnarray*}
    \P( \tau=\tau^i \in I_m)
    &\leq & \P( E_m^i \cap \bigcap_{k=1}^{m-1} N_k)
\\
    &\leq & \P( E_m^i) \cdot \P(\bigcap_{k=1}^{m-1} N_k) + e^{-\amu\ell}
\\
    &\leq &  \ldots
\\
    &\leq & \P( E_m^i) \cdot \prod_{k=1}^{m-1} \P(N_k) + m e^{-\amu\ell}
\\
    &= & \P( E_1^i) \cdot \P(N_1)^{m-1} + m e^{-\amu\ell}
\\
    &\leq & \P( E_1^i) \cdot \P(N_1)^{m-1} + M e^{-\amu\ell}.
\end{eqnarray*}

This yields
\begin{eqnarray}
     \sum_{m=1}^M \P( \tau=\tau^i \in I_m)
     &\leq&  \P( E_1^i) \cdot \sum_{m=1}^M \P(N_1)^{m-1} + M^2 e^{-\amu\ell} \notag
\\
     &\leq&  \P( E_1^i) \cdot\frac{1}{1-\P(N_1)} + M^2 e^{-\amu\ell}.
     \label{eqn:maintermseq6.3cont}
\end{eqnarray}

The last term is easily seen to tend to zero:
$$
    M^2 e^{-\amu\ell} = \frac{\mathcal{M}^2 T_\ast^2}{(L+\ell)^2} \, e^{-\amu B \sigma^{-2} }
    = \frac{\mathcal{M}^2 \sigma^{4/\alpha-2} \exp( 2(\sigma v)^{-2} /2 ) }
      {(L+\ell)^2} \, e^{-\amu B \sigma^{-2} },
$$
which tends to zero as $\sigma\to 0$, if we choose $B$ is sufficiently large ($\amu B>v^{-2}$).

In order to treat the first term in \eqref{eqn:maintermseq6.3cont}, first observe that
\begin{equation} \label{eqn:maintt7.1}
    1-\P( N_1) = \P(N_1^c) = \P( \bigcup_{j=1}^d E_1^j ) \geq  \sum_{j=1}^d \P( E_1^j )
    - \sum_{j_1,j_2\in\{1,\ldots,d\}, j_1\neq j_2} \P( E_1^{j_1}\cap  E_1^{j_2}),
\end{equation}
where the inequality follows from the inclusion-exclusion principle (also called Bonferroni's inequality).

Using again the Pickands lemma (Lemma~\ref{lem:pickands}), we have
\begin{equation} \label{eqn:maintt7.2}
        \P(E_1^j) \sim \frac{A_j^{1/\alpha} \mathcal{H}_{\alpha}}{\sqrt{2\pi}}
        \cdot L \cdot (\sigma v)^{1-2/\alpha} \, \exp( -(\sigma v)^{-2}/2 ),\qquad
        j\in\{1,\ldots,d\}.
\end{equation}

We shall further prove that for $j_1\neq j_2$, $j_1,j_2\in\{1,\ldots,d\}$,
\begin{equation} \label{eqn:admit7.3}
       \P(E_1^{j_1}\cap E_1^{j_2}) \leq \exp( - \eta (\sigma v)^{-2}/2 (1+o(1)) ),
\end{equation}
as $\sigma\to 0$, where $\eta>1$ is some constant. Inserting this into \eqref{eqn:maintt7.1}
and using \eqref{eqn:maintt7.2}, we obtain
\begin{equation} \label{eqn:maintt7.1-finish}
    1 - \P(N_1) \geq \sum_{j=1}^d \P( E_1^j ) (1+o(1)),
\end{equation}
because the leading order of $\sum_{j=1}^d \P( E_1^j )$ is given by
$\exp(-(\sigma v)^{-2}/2 (1+o(1)))$, by \eqref{eqn:maintt7.2}, while  the second term
in \eqref{eqn:maintt7.1} is of order $\exp(-\eta(\sigma v)^{-2}/2(1+o(1)))$
due to \eqref{eqn:admit7.3}, and thus of lower order as $\eta>1$.

Putting this together with \eqref{eqn:maintermseq6.3cont} shows that, as $\sigma\to 0$,
\begin{eqnarray*}
    \sum_{m=1}^M \P( \tau=\tau^i \in I_m)
    &\leq&
    \P( E_1^i) \cdot \frac{1}{1-\P(N_1)} + o(1)
\\
    &=& \frac{\frac{A_i^{1/\alpha} \mathcal{H}_{\alpha}}{\sqrt{2\pi}} \cdot L \cdot
        (\sigma v)^{1-2/\alpha} \, \exp( -(\sigma v)^{-2}/2 )  }
        { \sum_{j=1}^d \frac{A_j^{1/\alpha} \mathcal{H}_{\alpha}}{\sqrt{2\pi}} \cdot
        L \cdot (\sigma v)^{1-2/\alpha} \, \exp( -(\sigma v)^{-2}/2 ) } + o(1)
\\
    &=& \frac{A_i^{1/\alpha}}{ \sum_{j=1}^d A_j^{1/\alpha}} + o(1),
\end{eqnarray*}
as required by \eqref{eqn:vspdistrib1upper}. It remains to prove \eqref{eqn:admit7.3}.

{\it Step 5: Proof of \eqref{eqn:admit7.3}.}

As a comment, note that \eqref{eqn:admit7.3} means that the probability of ``joint exits''
of different components in one of the $I_m$ intervals is of lower order
(compared to the exit of only one component).

To simplify the notation, let $j_1=1$, $j_2=2$. By definition,
$$
     E_1^1\cap E_1^2 = \left\lbrace \max_{s\in[0,L]} \VVV^1_s \geq \sigma^{-1},
     \max_{t\in[0,L]} \VVV_t^2 \geq \sigma^{-1}\right\rbrace.
$$
Obviously,
\[
      E_1^1\cap E_1^2 \subseteq   \left\lbrace \max_{s,t\in[0,L]}
         (\VVV_s^1 +  \VVV_t^2) \geq 2\sigma^{-1}\right\rbrace
         = \left\lbrace \max_{s,t\in[0,L]}  \Lambda(s,t)
           \geq 2 \sigma^{-1}\right\rbrace,
\]
where $\Lambda(s,t):=\VVV_s^1 + \VVV_t^2$ is a centered
Gaussian process with index set $\mathcal{T}:=[0,L]^2$.

According to the general theory \cite{L}, the evaluation of the large values' probability
for $\Lambda$ requires a bound for its largest variance and the estimation of the
corresponding {\it  covering numbers}.
By Assumption (V3), for all $s,t\in \R$ we have the variance bound
$$ 
     \E \Lambda(s,t)^2 =  2v^2 (1+ \E[\VVV_s^1 \VVV_t^2]) \le 2v^2(1+ \bar\rho)
$$ 
with some $\bar\rho<1$. Next, recall that for a Gaussian process $(X_t, t\in T)$
the covering number $\NN(X,T, \eps)$ is the minimal number of balls needed to cover
the set $T$ with respect to {\it the Dudley semi-metric}
$$
    \rho_X(s,t)^2:= \E[ (X_s-X_t)^2], \qquad s,t\in T.
$$
We refer to Section~14 of \cite{L} for other relevant definitions.

Due to the property (V2) we have
\[
  \rho_{\VVV^i}(t,t+h)^2 = 2v^2A_i |h|^\alpha (1+o(1)), \qquad \textrm{as } h\to 0.
\]
It follows that
\[
  \NN(\VVV^i,[0,L],\eps) \le c_1 L \eps^{-2/\alpha}, \qquad \eps>0, L\ge 1.
\]
Hence,
\[
  \NN(\Lambda,\mathcal{T}, 2\eps) \le c_1^2 L^2 \eps^{-2/\alpha}, \qquad \eps>0, L\ge 1.
\]
Furthermore, for {\it the Dudley integral} we have the estimate
\[
	\int_0^{2v} \sqrt{\ln \NN(\Lambda,\mathcal{T},\eps)} \ d\eps
   \le c_2 \sqrt{\ln L}, \qquad L\ge 2.
\]
We can finally apply Corollary~2 from  \cite[Section~14, p.181] {L}, which
involves the Dudley integral and shows in our case that (with $\bar \Phi$
the tail of the standard normal distribution) for small enough $\sigma$
\begin{eqnarray*}
    && \P\left(  \max_{(s,t) \in \mathcal{T}}
       \Lambda(s,t) \geq 2 \sigma^{-1}\right)
\\
    &\leq& \bar \Phi( (2\sigma^{-1} - c_3 \sqrt{\ln L} ) / (v\sqrt{2(1+\bar \rho)}) )
\\
    &\leq& \exp(- \frac{\sigma^{-2}}{v^2 (1+\bar \rho)} (1+o(1)) )
\\
    &=& \exp( - \eta (\sigma v)^{-2}/2 (1+o(1)) ),
\end{eqnarray*}
with $\eta:=2/(1+\bar \rho)>1$ as required in \eqref{eqn:admit7.3}.

{\it Step 6: We prove the second assertion of the lemma, i.e.\ the limit theorem for $\tau^i$,
i.e.\ $\eqref{eqn:taulimitepszero}$.}

We recall that we obtained in \eqref{eqn:copyforlimittheorem} that for any $\mathcal{M}>0$
$$
    \limsup_{\sigma\to 0} \P( \tau^i \geq \mathcal{M} T_\ast )
    \leq \exp(- \mathcal{M} c_{\alpha,A_i,v} ),
$$
where $c_{\alpha,A_i,v}=A_i^{1/\alpha} \mathcal{H}_\alpha v^{1-2/\alpha}/\sqrt{2\pi}$.
Using the specific choice $\mathcal{M}:=t c_{\alpha,A_i,v}^{-1}$, we obtain that, for any $t>0$,
$$
    \limsup_{\sigma\to 0} \P( \tau^i \geq t c_{\alpha,A_i,v}^{-1} T_\ast )
    \leq \exp(- t ).
$$
This already shows one of the estimates needed for the limit theorem \eqref{eqn:taulimitepszero}.

We now proceed with proving the lower bound in \eqref{eqn:taulimitepszero}. For this purpose,
let $\mathcal{M}>0$ and set $M$ (and all other variables) as in Step 1. Then
\begin{eqnarray*}
    \P( \tau^i \geq \mathcal{M} T_\ast )
    &\geq & \P( \bigcap_{m=1}^M (E_m^i)^c \cap \{ \tau^i \not\in \bigcup_{m=1}^M J_m\})
\\
    &\geq & \P( \bigcap_{m=1}^M (E_m^i)^c) - \P(\tau^i \in \bigcup_{m=1}^M J_m)
\\
    &\geq & \prod_{m=1}^M \P( (E_m^i)^c)  -  M e^{-\amu\ell} - \P( \tau^i \in \bigcup_{m=1}^M J_m ) ,
\\
     &= & (1- \P( E_1^i))^M  -  M e^{-\amu\ell}-\P( \tau^i \in \bigcup_{m=1}^M J_m ),
\end{eqnarray*}
by the decoupling argument from Lemma~\ref{lem:nonstr.4.2} and using stationarity.
The term $Me^{-\amu\ell}$ is seen to tend to zero with $\sigma\to 0$,
as in \eqref{eqn:useforlowerboundtaus}, by the choice of the constant $B$. Further,
$\P( \tau^i \in \bigcup_{m=1}^M J_m )$ tends to zero, as observed in Step 2.

Since $\P(E_1^i)\to 0$ as $\sigma\to 0$, we obtain
$$
   \P( \tau^i \geq \mathcal{M} T_\ast ) \geq e^{- M\P( E_1^i) (1+o(1))} - o(1).
$$
We further know from \eqref{eqn:saveforlimt59} that
$M\P(E_1^i) \sim \mathcal{M} \cdot c_{\alpha,A_i,v}$, which implies
$$
    \liminf_{\sigma\to 0}\P( \tau^i \geq \mathcal{M} T_\ast )
    =\exp(- \mathcal{M} \cdot c_{\alpha,A_i,v}).
$$
Using the specific choice $\mathcal{M}:=t c_{\alpha,A_i,v}^{-1}$, we obtain that,
for any $t>0$,
$$
    \liminf_{\sigma\to 0} \P( \tau^i \geq t c_{\alpha,A_i,v}^{-1} T_\ast )
    \geq \exp(- t ),
$$
which completes the proof of \eqref{eqn:taulimitepszero}.

{\it Step 7: We prove the last assertion, i.e.\ the limit theorem for $\tau$,
i.e.\ $\eqref{eqn:taulimitepszero-taus}$.}

The proof is very similar to the bounds for the $\tau^i$. First,
\begin{eqnarray*}
\P( \tau \geq \mathcal{M} T_\ast )
&\leq & \P( \bigcap_{m=1}^M \bigcap_{j=1}^d (E_m^j)^c )
\\
&\leq & \prod_{m=1}^M \P( \bigcap_{j=1}^d (E_m^j)^c)  +  M e^{-\amu\ell}
\\
&= & (1- \P( \bigcup_{j=1}^d E_1^j))^M + M e^{-\amu\ell}
\\
&\leq & \exp(- M\P(\bigcup_{j=1}^d E_1^j)) + M e^{-\amu\ell}.
\end{eqnarray*}
Again the second term is of lower order (cf.\ \eqref{eqn:useforlowerboundtaus}),
while for the first -- by the argument
in \eqref{eqn:maintt7.1}-\eqref{eqn:maintt7.1-finish} -- one can see that
$$
    \P(\bigcup_{j=1}^d E_1^j) = \sum_{j=1}^d \P( E_1^j ) (1+o(1)).
$$

This shows (using \eqref{eqn:maintt7.1} and the definition of $T_\ast$)
$$
    \limsup_{\sigma\to 0} \P( \tau \geq \mathcal{M} T_\ast )
    \leq \exp( - \mathcal{M} \sum_{j=1}^d c_{\alpha,A_j,v} ).
$$
Using the specific choice $\mathcal{M}:= t (\sum_{j=1}^d c_{\alpha,A_j,v})^{-1}$
shows the upper bound in  \eqref{eqn:taulimitepszero-taus}.

For the lower bound, one can also proceed analogously:
\begin{eqnarray*}
    && \P( \tau \geq \mathcal{M} T_\ast )
\\
    &\geq & \P( \bigcap_{m=1}^M \bigcap_{j=1}^d (E_m^j)^c \cap
    \{ \forall j : \tau^j \not\in \bigcup_{m=1}^M J_m\})
\\
    &\geq & \P( \bigcap_{m=1}^M \bigcap_{j=1}^d (E_m^j)^c )
            - \P( \exists j : \tau^j \in \bigcup_{m=1}^M J_m)
\\
    &\geq & \prod_{m=1}^M \P( \bigcap_{j=1}^d (E_m^j)^c)
    -  M e^{-\amu\ell}-\sum_{j=1}^d \P( \tau^j \in \bigcup_{m=1}^M J_m)
\\
    &= & (1- \P( \bigcup_{j=1}^d E_1^j))^M - M e^{-\amu\ell}
        -\sum_{j=1}^d \P( \tau^j \in \bigcup_{m=1}^M J_m)
\\
    &\geq & \exp(- M\P(\bigcup_{j=1}^d E_1^j) (1+o(1)) )
         - M e^{-\amu\ell}-\sum_{j=1}^d \P( \tau^j \in \bigcup_{m=1}^M J_m).
\end{eqnarray*}
Again the second and third term vanish (cf.\ \eqref{eqn:useforlowerboundtaus} and Step 2).
For the first expression, we can obtain
$$
    \P(\bigcup_{j=1}^d E_1^j) = \P( \max_{j\in\{1,\ldots,d\} }
    \max_{t\in[0,L]} \VVV_t^j \geq \sigma^{-1})
    \leq \sum_{j=1}^d \P( \max_{t\in[0,L]} \VVV_t^j \geq \sigma^{-1}),
$$
and we apply the Pickands lemma (see \eqref{eqn:maintt7.2}) to this giving
$$
    \liminf_{\sigma\to 0} \P( \tau \geq \mathcal{M} T_\ast ) \geq \exp( - \mathcal{M}
    \sum_{j=1}^d c_{\alpha,A_j,v} ).
$$
Again, using the specific choice $\mathcal{M}:= t (\sum_{j=1}^d c_{\alpha,A_j,v})^{-1}$
shows the lower bound in \eqref{eqn:taulimitepszero-taus}.

\subsection{Proof of Proposition~\ref{prop:veryslowpullingstat}}
In this section, we prove Proposition~\ref{prop:veryslowpullingstat}.
The goal is to apply Lemma~\ref{lem:veryslowpullingstatlemma} with $\alpha=1$ and, accordingly,
$\mathcal{H}_\alpha=1$. We only have to show that the processes $(\tV^i_t)$ defined
in $\eqref{eqn:Vti}$ satisfy the assumptions (V1)-(V4) of Lemma~\ref{lem:veryslowpullingstatlemma}.

Note that (V1) and (V2) follow directly from relation \eqref{eqn:lemma7directv1} in
Lemma~\ref{lem:7directv1}, where $\alpha=1$ and $v$ and the $A_i$ are as
in the statement of the proposition.

Let us prove (V3). First note that it is sufficient to consider $s=0$ and $t\geq 0$,
by stationarity. Fix $i\neq j$. Note that
$$
    | \E[ \tV_t^i \tV_0^j] | \leq \sqrt{ \E[ (\tV_t^i)^2]\cdot \E[ (\tV_0^j)^2] } = v^2,
$$
with equality if and only if the two random variables are linearly dependent. In the latter case,
we would have a.s.\ $\tV_t^i = c \tV_0^j$, which has probability zero (for any $t>0$).
Therefore, $\E[ \tV_t^i \tV_0^j] < v^2$ for all $t>0$.

At $t=0$, by \eqref{eqn:transfertov3lemma7} in Lemma~\ref{lem:7directv1},
we have $\E[ \tV_0^i \tV_0^j]=-1/(2d)$, so that for small $t$ we can bound
$\E[ \tV_t^i \tV_0^j]$ away from $0$, by continuity. Further,
$\E[ \tV_t^i \tV_0^j] = \sum_{k=1}^d a_k^{i,j} e^{\lambda_k t}$ is a sum of
exponential terms, because each $\tV^i$ is a linear combination of independent
Ornstein--Uhlenbeck processes, so that, for large $t$ it
tends to zero. Summarizing, the function $t\mapsto \E[ \tV_t^i \tV_0^j]$ is continuous,
negative for small $t$, never reaches $v^2$, and tends to zero for $t\to\infty$.
Therefore, we can bound it
$$
    \E[ \tV_t^i \tV_0^j] \leq \bar\rho_{i,j} v^2
$$
with $\bar\rho_{i,j}<1$. Taking $\bar\rho:=\max_{i\neq j} \bar \rho_{i,j}< 1$
then satisfies (V3).

The fact that (V4) holds for $(\tV_t^i)_{1\le i\le d}$ with $\mu:=\min_j |\lambda_j|$ follows from
\eqref{eqn:covarianceU}, as we noticed in the proof of Lemma~\ref{lem:nonstr.4.2concrete}.

\subsection{Proof of Theorem~\ref{thm:veryslowpulling}}

We shall prove Theorem~\ref{thm:veryslowpulling} by reducing it to the stationary case considered in
Proposition \ref{prop:veryslowpullingstat}.

\subsubsection{Limit theorem for $\tau^i$}

First of all notice that the slow pulling condition \eqref{eqn:condveryslowpulling} implies
\be \label{eqn:epsreallysmall}
  \eps \ll \exp(-K\sigma^{-2}) \qquad \forall K>0.
\ee
We will only need this (slightly weaker) condition.

\paragraph{Upper bound for the exit time.}
Recall that the sufficient condition for a break in the chain element $i$ at time $t$ is
given in \eqref{eqn:exitsufficientshort}; it follows that the simpler condition
\[
   \sigma \tV_t^i\ge 1 +\eps K^i + \sigma \Xi e^{-\amu t}
\]
is also sufficient for the break.

Let us denote by
\[
  \theta_\sigma^i:=\sigma v \exp((\sigma v)^{-2}/2) \frac{\sqrt{2\pi}}{A_i}
\]
the norming factor from our limit theorem. Then for all $r,\delta,F>0$ we have
\begin{eqnarray*}
  \P(\tau^i\ge r \theta_\sigma^i ) &\leq&
  \P\left( \sigma \tV_t^i < 1 +\eps K^i + \sigma \Xi e^{-\amu t},
  0 \le t \le r \theta_\sigma^i\right)
\\
   &\leq&
  \P\left(\sigma \tV_t^i < 1 +\eps K^i + \sigma F e^{-\amu t},
  0 \le t \le r \theta_\sigma^i\right) +\P(\Xi>F)
\\
   &\leq&
  \P\left(\sigma \tV_t^i < 1 +\eps K^i + \sigma F e^{-\amu \delta \theta_\sigma^i},
      \delta\theta_\sigma^i \le t \le  r\theta_\sigma^i \right) +\P(\Xi>F)
\\
  &=&
  \P\left( \sigma_{i*} \tV_t^i < 1, \  \delta \theta_\sigma^i \le t \le r\theta_\sigma^i \right)
  + \P(\Xi>F),
\end{eqnarray*}
where
$\sigma_{i*}= \sigma ( 1 +\eps K^i + \sigma F e^{-\amu \delta \theta_\sigma^i})^{-1}$.
Using \eqref{eqn:epsreallysmall}, notice that
$\eps K^i + \sigma F e^{-\amu \delta \theta_\sigma^i}\ll \sigma^2$, hence
$\sigma_{i*}^{-2}=\sigma^{-2}+o(1)$, thus providing
$\theta_{\sigma_{i*}}^i \sim \theta_\sigma^i$, and we obtain eventually
\begin{eqnarray*}
   \P(\tau^i\ge \theta_\sigma^i r) &\leq&
   \P\left(\sigma_{i*} \tV_t^i < 1, \
      2\delta\theta_{\sigma_{i*}}^i \le t \le (r-\delta)\theta_{\sigma_{i*}}^i\right)
      + \P(\Xi>F)
 \\
 &\leq&
   \P\left(\sigma_{i*} \tV_t^i < 1,\ 0\le t \le (r-\delta)\theta_{\sigma_{i*}}^i\right)
\\
     &&     + \P\left(\exists t\in [0, 2\delta \theta_{\sigma_{i*}}^i] : \sigma_{i*} \tV_t^i\ge 1 \right)
     + \P(\Xi>F).
\end{eqnarray*}
By applying Proposition \ref{prop:veryslowpullingstat} with $\sigma_{i*}$
in place of $\sigma$ we have that  under \eqref{eqn:epsreallysmall}
\[
   \limsup_{\eps\to 0,\sigma\to 0}  \P(\tau^i \ge r \theta_\sigma^i)
   \leq  e^{-(r-\delta)} +[1- e^{-2\delta}] + \P(\Xi>F).
\]
Finally, by letting $\delta\to 0$, $F\to+\infty$ we obtain under \eqref{eqn:epsreallysmall}
\[
   \limsup_{\eps\to 0,\sigma\to 0}  \P(\tau^i\ge r \theta_\sigma^i) \leq e^{-r},
\]
as desired.

\paragraph{Lower bound for the exit time.}
Recall that the {\it necessary} condition for the break is given in \eqref{eqn:exitnecessaryshort}.
Assuming as before that $\Xi\le F$, we obtain the necessary condition
\[
   \sigma \tV_t^i\ge 1 - \frac{\eps t}{d}- \eps K^i - \sigma F e^{-\amu t}.
\]
We handle this condition differently for relatively small $t$ and for larger $t$.
Namely, for $t\in[0, 3|\ln\sigma|/\amu]$ we use that (under \eqref{eqn:epsreallysmall})
\[
    \sup_{t\in[0,3|\ln\sigma|/\amu]}
        \left[ \frac{\eps t}{d} +\eps K^i + \sigma F e^{-\amu t}\right]
 \le
   \frac{\eps}{d} \cdot  \frac{3|\ln\sigma|}{\amu}  +\eps K^i + \sigma F \to 0.
\]
Therefore, for $t\in[0,3|\ln\sigma|/\amu]$ we obtain the necessary condition
\[
   \sigma \tV_t^i\ge \frac{1}{2}.
\]
On the other hand, for $t\in [3|\ln\sigma|/\amu, r \theta_\sigma^i]$ we obtain
the necessary condition
\[
   \sigma \tV_t^i\ge 1 - \frac{\eps r \theta_\sigma^i }{d}- \eps K^i -  F\sigma^{4},
\]
which can be rewritten as  $\sigma_{i*} \tV_t^i\ge 1$, where
$\sigma_{i*}:=\sigma\left( 1 - \frac{\eps r \theta_\sigma^i}{d}- \eps K^i - F\sigma^{4}\right)^{-1}$
again satisfies the requirement
$\sigma_{i*}^{-2}=\sigma^{-2}+o(1)$ (under \eqref{eqn:epsreallysmall}), hence,
$\theta_{\sigma_{i*}}^i \sim \theta_\sigma^i$.
We obtain
\begin{eqnarray*}
  && \P(\tau^i\ge r \theta_\sigma^i)
\\
  &\geq&
  \P(\sigma_{i*} \tV_t^i < 1, 0\le t\le r \theta_\sigma^i ) - \P(\Xi>F)
  - \P\left( \sigma \sup_{t\in[0,3|\ln\sigma|/\amu]} \tV_t^i\ge \frac{1}{2} \right)
\\
&\geq&
\P(\sigma_{i*} \tV_t^i < 1, 0\le t\le (1+\delta) r \theta_{\sigma_{i*}}^i) - \P(\Xi>F)
\\
 &&  - \P\left( \sigma \sup_{t\in[0,3|\ln\sigma|/\amu]} \tV_t^i\ge \frac{1}{2} \right).
\end{eqnarray*}
By stationarity and Gaussianity of $\tV^i$ we have
\begin{equation} \label{eqn:gaussiansupremum}
  \lim_{\sigma\to 0}  \P\left( \sigma \sup_{t\in[0,3|\ln\sigma|/\amu]} \tV_t^i\ge \frac{1}{2} \right)
  = 0.
\end{equation}
Indeed, since the process $\tV^i$ is bounded, by \cite[Theorem 1, Section 12]{L} for every $h>0$
and all $\sigma$ small enough one has
\[
  \P\left( \sigma \sup_{t\in[0,1]} \tV_t^i \ge \frac{1}{2} \right)
  \le \exp\left(- \frac{(\sigma v)^{-2}}{8(1+h)}\right).
\]
Hence, using stationarity, for every $L>0$ we see that
\[
  \P\left( \sigma \sup_{t\in[0,L]} \tV_t^i \ge \frac{1}{2} \right)
  \le (L+1) \exp\left(- \frac{(\sigma v)^{-2}}{8(1+h)} \right).
\]
Applying this with $L= 3|\ln\sigma|/\amu$ we obtain \eqref{eqn:gaussiansupremum}.

By using \eqref{eqn:gaussiansupremum} and applying Proposition \ref{prop:veryslowpullingstat} with $\sigma_{i*}$
in place of $\sigma$ we have under \eqref{eqn:epsreallysmall}
\[
   \liminf_{\eps\to 0,\sigma\to 0}  \P(\tau^i\ge r \theta_\sigma^i) \geq e^{-(1+\delta)r}  - \P(\Xi>F).
\]
Finally, by letting $\delta\to 0$, $F\to+\infty$ we obtain that  under \eqref{eqn:epsreallysmall}
\[
   \liminf_{\eps\to 0,\sigma\to 0} \P(\tau^i\ge r \theta_\sigma^i) \geq e^{-r},
\]
as desired.

\subsubsection{Limit theorem for $\tau$}

Let us denote by
\[
  \theta_\sigma:=\sigma v \exp((\sigma v)^{-2}/2) \frac{\sqrt{2\pi}}{2d}
\]
the norming factor from our limit theorem.

\paragraph{Upper bound for the total exit time.}
Using the same argument as in the previous upper bound we obtain
for all $r,\delta,F>0$
\begin{eqnarray*}
  \P(\tau\ge r \theta_\sigma ) &=&  \P(\tau^i \ge r \theta_\sigma, 1\le i\le d )
\\
  &\leq&
  \P\left( \sigma \tV_t^i < 1 +\eps K^i + \sigma \Xi e^{-\amu t},
  0 \le t \le r \theta_\sigma, 1\le i\le d \right)
\\
   &\leq&
  \P\left(\sigma \tV_t^i < 1 +\eps K^i + \sigma F e^{-\amu t},
  0 \le t \le r \theta_\sigma, 1\le i\le d \right)
\\
   &&   +\P(\Xi>F)
\\
   &\leq&
  \P\left(\sigma \tV_t^i < 1 +\eps K^i + \sigma F e^{-\amu \delta \theta_\sigma},
      \delta\theta_\sigma \le t \le  r\theta_\sigma, 1\le i\le d\right)
\\
   &&   +\P(\Xi>F)
\\
  &\leq&
  \P\left( \sigma_{*} \tV_t^i < 1, \  \delta\theta_\sigma \le t \le r\theta_\sigma, 1\le i\le d \right)
  + \P(\Xi>F),
\end{eqnarray*}
where
$\sigma_{*}= \sigma ( 1 +\eps  \max_{1\le i\le d}K^i + \sigma F e^{-\amu \delta \theta_\sigma})^{-1}$.
Using \eqref{eqn:epsreallysmall}, we obtain again
$\sigma_{*}^{-2}=\sigma^{-2}+o(1)$, providing $\theta_{\sigma_{*}} \sim \theta_\sigma$, and we obtain eventually
\begin{eqnarray*}
   \P(\tau\ge  r \theta_\sigma) &\leq&
   \P\left(\sigma_{*} \tV_t^i < 1, \
      2\delta\theta_{\sigma_{*}} \le t \le (r-\delta)\theta_{\sigma_{*}}, 1\le i\le d  \right)
      + \P(\Xi>F)
 \\
 &\leq&
   \P\left(\sigma_{*} \tV_t^i < 1,\ 0\le t \le (r-\delta)\theta_{\sigma_{*}}, 1\le i\le d \right)
\\
     &&     + \P\left(\exists t\in [0, 2\delta \theta_{\sigma_{*}} \sigma_{*}],
     \exists i\in \{1,\dots,d\}  : \sigma_{*}\tV_t^i\ge 1 \right)   + \P(\Xi>F).
\end{eqnarray*}
By applying Proposition \ref{prop:veryslowpullingstat} with $\sigma_{*}$
in place of $\sigma$ we have
\[
   \limsup_{\eps\to 0,\sigma\to 0}  \P(\tau \ge r \theta_\sigma)
   \leq  e^{-(r-\delta)} +[1- e^{-2\delta}] + \P(\Xi>F).
\]
Finally, by letting $\delta\to 0$, $F\to+\infty$ we obtain that  under \eqref{eqn:epsreallysmall}
\[
   \limsup_{\eps\to 0,\sigma\to 0}  \P(\tau\ge r \theta_\sigma) \leq e^{-r},
\]
as desired.

\paragraph{Lower bound for the total exit time.}
Using the same argument as in the previous lower bound we obtain
for all $r,\delta,F>0$
\begin{eqnarray*}
    \P(\tau\ge r \theta_\sigma )
	&=&  \P(\tau^i\ge r \theta_\sigma, 1\le i\le d)
\\
  &\geq&   \P(\sigma_{i*} \tV_t^i < 1, 0\le t\le r \theta_\sigma, 1\le i\le d ) - \P(\Xi>F)
\\
  &&  - \P\left( \sigma \max_{1\le i\le d} \sup_{t\in[0,3|\ln\sigma|/\amu]} \tV_t^i\ge \frac{1}{2} \right)
\\
  &\geq&
  \P( \sigma_* \tV_t^i < 1, 0\le t\le (1+\delta) r \theta_{\sigma_{*}}, 1\le i\le d) - \P(\Xi>F)
\\
 &&  - \P\left( \sigma \max_{1\le i\le d} \sup_{t\in[0,3|\ln\sigma|/\amu]} \tV_t^i\ge \frac{1}{2} \right)
\end{eqnarray*}
where $ \sigma_*:=\max_{1\le i\le d} \sigma_{i*}$.
Again we have
\begin{eqnarray*}
  && \lim_{\sigma\to 0}  \P\left( \sigma  \max_{1\le i\le d} \sup_{t\in[0,3|\ln\sigma|/\amu]} \tV_t^i\ge \frac{1}{2} \right)
\\
  &\le& \sum_{i=1}^{d}  \lim_{\sigma\to 0}  \P\left( \sigma \sup_{t\in[0,3|\ln\sigma|/\amu]} \tV_t^i\ge \frac{1}{2} \right)
  = 0.
\end{eqnarray*}
By applying Proposition \ref{prop:veryslowpullingstat} with $\sigma_{*}$
in place of $\sigma$ we have  that under \eqref{eqn:epsreallysmall}
\[
   \liminf_{\eps\to 0,\sigma\to 0}  \P(\tau\ge r \theta_\sigma) \geq e^{-(1+\delta)r}  - \P(\Xi>F).
\]
Finally, by letting $\delta\to 0$, $F\to+\infty$ we obtain that  under \eqref{eqn:epsreallysmall}
\[
   \liminf_{\eps\to 0,\sigma\to 0}  \P(\tau\ge r \theta_\sigma) \geq e^{-r},
\]
as desired.

\subsubsection{Limit theorem for the break position}

Let us fix $i\in\{1,\dots,d\}$ and $\delta,r,F>0$.
Then by using the necessary and the sufficient conditions for exit, we have
\begin{eqnarray*}
   && \P(\tau=\tau^i)
	\\
	&\ge&
\P(\tau\ge \delta \theta_\sigma^i, \tau^i\in [\delta \theta_\sigma^i,r\theta_\sigma^i],
\\
    && \textrm{and } \sigma V_s^j <1-\frac{\eps s}{d}- \eps K^j, \delta \theta_\sigma^i \le s \le \tau^i , j\not= i)
\\
   &\geq& \P\Big( \tau\ge \delta \theta_\sigma^i,  \exists t\in [\delta \theta_\sigma^i, r  \theta_\sigma^i]:
   \sigma\tV_t^i\ge 1+ \eps K^i + \sigma\Xi e^{-\amu t}
\\
   &&  \textrm{and } \sigma\tV_s^j < 1 -\frac{\eps s}{d}- \eps K^j - \sigma\Xi e^{-\amu s},
       \delta \theta_\sigma^i  \le s\le t, j\not=i \Big)
\\
   &\geq& \P\Big( \exists t\in [\delta \theta_\sigma^i, r  \theta_\sigma^i]:
   \sigma\tV_t^i\ge 1+ \eps K^i + \sigma\Xi e^{-\amu t}
\\
   &&  \textrm{and } \sigma\tV_s^j < 1 -\frac{\eps s}{d}- \eps K^j - \sigma\Xi e^{-\amu s},
       0\le s\le t, j\not=i \Big) - \P(\tau\le \delta \theta_\sigma^i)
\\
   &\geq& \P\Big( \exists t\in [\delta \theta_\sigma^i, r \theta_\sigma^i]:
        \sigma\tV_t^i\ge 1 + \eps K^i + \sigma F e^{-\amu \delta \theta_\sigma^i}
\\
   &&  \textrm{and } \sigma\tV_s^j < 1 - \frac{\eps r \theta_\sigma^i }{d} - \eps K^j
      - \sigma F e^{-\amu \delta \theta_\sigma^i },
      \delta \theta_\sigma^i \le s\le t, j\not=i \Big)
\\
   &&      -  \P(\tau\le \delta \theta_\sigma^i) - \P(\Xi>F)
\\
   &=& \P\Big( \exists t\in [\delta \theta_\sigma^i, r \theta_\sigma^i]:
   \sigma_{*i} \tV_t^i\ge 1
    \ \textrm{and } \sigma_{*j} \tV_s^j< 1,
   \delta \theta_\sigma^i \le s\le t, j\not=i \Big)
\\
  &&   - \P(\tau\le \delta \theta_\sigma^i) - \P(\Xi>F),
\end{eqnarray*}
where $\sigma_{*i}:=\sigma(1+\eps K^i + \sigma F e^{-\amu \delta \theta_\sigma^i})^{-1}$
and  $\sigma_{*j}=\sigma(1 - \frac{\eps r \theta_\sigma^i }{d} - \eps K^j
     - \sigma F e^{-\amu \delta \theta_\sigma^i })^{-1}$ for $j\not=i$. By \eqref{eqn:epsreallysmall}, $\sigma_{*j}^{-2}=\sigma^{-2}+o(1)$.

Introduce the corresponding exit times
$\tau_*^j=\inf \{t: \sigma_{*j} \tV_t^j\ge 1\}$, $1\le j\le d$,
and $\tau_*:=\min_{1\le j\le d} \tau_*^j$.
Using this notation, we have
\begin{eqnarray*}
 &&  \P\Big( \exists t\in [\delta \theta_\sigma^i, r \theta_\sigma^i]:
   \sigma_{*i} \tV_t^i\ge 1
    \ \textrm{and } \sigma_{*j} \tV_s^j< 1,
   \delta \theta_\sigma^i \le s\le t, j\not=i \Big)
\\
  &\geq& \P( \ttau_*=\ttau_{*}^i\in [\delta \theta_\sigma^i, r \theta_\sigma^i])
\\
  &\geq& \P( \ttau_*=\ttau_{*}^i)
  - \P(\ttau_*^i\not\in [\delta \theta_\sigma^i, r \theta_\sigma^i]).
\end{eqnarray*}

By using Remark \ref{rem:sigmastar} for the exit times (first step) and the limit theorem for
$\tau$ that is already proved (second step), we have that under \eqref{eqn:epsreallysmall}
\begin{eqnarray*}
  \liminf_{\eps\to 0,\sigma\to 0} \P( \tau = \tau^i)
   &\geq& \lim_{\sigma\to 0} \P( \ttau_\ast = \ttau_\ast^i)
    - ( 1 - e^{-\delta} + e^{-r} )
    - \lim_{\eps\to 0,\sigma\to 0}  \P(\tau \le \delta \theta_\sigma^i)
\\
  &&  - \P(\Xi>F)
\\
 &=& \lim_{\sigma\to 0} \P( \ttau_\ast = \ttau_\ast^i)
    - ( 1 - e^{-\delta}+ e^{-r}  )
    -  (1-e^{- 2d\delta/A_i})
\\
  &&    - \P(\Xi>F).
\end{eqnarray*}

Then, letting $\delta\to 0$, $r,F\to\infty$ and using Remark \ref{rem:sigmastar} for the exit positions
we obtain the lower bound
\[
  \liminf_{\eps\to 0,\sigma\to 0} \P(\tau=\tau^i)
  \ge \lim_{\sigma\to 0} \P( \ttau_*=\ttau_{*}^i) =
  \begin{cases}
       \frac{1}{d-1} & i\in\{2,\ldots,d-1\};\\
       \frac{1}{2(d-1)} & i\in\{1,d\}.
  \end{cases}
\]
Since the right hand sides of the lower bound sum up in $i$ to one,
the corresponding upper bound follows automatically,
which proves the result.


\section{Moderately slow pulling regime}
\label{sec:moderate}
This regime is defined by the following two conditions
\be \label{eqn:moderate1}
  \eps\to 0, \ \sigma\to 0, \  h:= \frac{\eps}{\sigma} \to 0,
\ee
but
\be \label{eqn:moderate2}
  \sigma^2 \ln\eps \to 0.
\ee

\subsection{Result for the stationary model}

We start with considering a simplified model including a stationary process
and will pass later to the initial setting.

Let $\tV_t^i$ be the stationary Gaussian processes defined in \eqref{eqn:Vti}.
Define the related exit times
\[
  \ttau^i:=\inf\left\{ t:  \sigma \tV^i_t \ge  \frac{t_\ast-t}{t_\ast}    \right\}
\]
and $\ttau:=\min_{1\le i\le d} \ttau^i$. The main result for the stationary case is as follows.

\begin{prop} \label{prop:moderatelyslowpullingstat}
Let assumptions $\eqref{eqn:moderate1}$ and $\eqref{eqn:moderate2}$ be verified. Then
\[
\P( \ttau=\ttau^i ) \to \begin{cases}
                         \frac{1}{d-1} & i\in\{2,\ldots,d-1\};\\
                         \frac{1}{2(d-1)} & i\in\{1,d\}.
                      \end{cases}
\]
Further, for any $i\in\{1,\ldots,d\}$,
\begin{equation} \label{eqn:moderateexittimescaling_i}
     h \sqrt{\ln(h^{-1})}
        \left(t_\ast-\ttau^i-\gamma h^{-1} \sqrt{\ln (h^{-1})} \right)
     \tod \chi_i,\qquad i\in \{1,\ldots,d\},
\end{equation}
and
\begin{equation} \label{eqn:moderateexittimescaling}
   h \sqrt{\ln (h^{-1})}
     \left(t_\ast-\ttau-\gamma  h^{-1} \sqrt{\ln (h^{-1})}\right)
     \tod \chi_0,
\end{equation}
where $\chi_i$ is a double exponential random variable with parameters
$a_i:=vd A_i  / \sqrt{2\pi}$ for $i\in\{1,\ldots,d\}$,
$a_0:=\sum_{i=1}^d a_i = 2vd^2 /\sqrt{2\pi}$ and $b:=\sqrt{2}/(vd)$,
while $v$, $\gamma$, and the $A_i$ are defined in $\eqref{eqn:parameterv0}$.
\end{prop}

\begin{rem}\label{rem:moderatesigmastar}
At some point we will need a minor extension of Proposition~\ref{prop:moderatelyslowpullingstat}
allowing a very mild flexibility of exit times. Namely, let $\sigma_{i*}$, $1\le i\le d$,
be so close to $\sigma$ that
$\tfrac{ \sigma_{*i}}{ \sigma} = 1+ o\left( \tfrac{1}{\ln(h^{-1})}\right)$.
Let
\[
  \ttau_*^i:=\inf\left\{ t:   \sigma_{i*} \tV^i_t \ge  \frac{t_\ast-t}{t_\ast}  \right\}
\]
and $\ttau_*:=\min_{1\le i\le d} \ttau_*^i$.

Then we still have
\begin{eqnarray} \label{eqn:moderatefacescaling}
&& \P( \ttau_*=\ttau_*^i ) \to
                      \begin{cases}
                         \frac{1}{d-1} & i\in\{2,\ldots,d-1\};\\
                         \frac{1}{2(d-1)} & i\in\{1,d\}.
                      \end{cases}
\\ \label{eqn:moderateexittimescaling_i2}
    &&  h \sqrt{\ln(h^{-1})}
        \left(t_\ast-\ttau_*^i-\gamma h^{-1} \sqrt{\ln (h^{-1})} \right)
     \tod \chi_i,\ i\in \{1,\ldots,d\},
\\  \notag
   &&   h \sqrt{\ln (h^{-1})}
     \left(t_\ast-\ttau_*-\gamma  h^{-1} \sqrt{\ln (h^{-1})}\right)
     \tod \chi_0,
\end{eqnarray}
where the $(\chi_i)$ are as above.
\end{rem}

\subsection{Proof of Proposition \ref{prop:moderatelyslowpullingstat}}
Let us denote $\psi:=\ln (h^{-1})$.
Fix an $r\in\R$. The main part of the proof consists in the evaluation of the probability
\begin{eqnarray*}
       \P\left( h\sqrt{\psi} \big( t_\ast-\ttau^i-\gamma h^{-1} \sqrt{\psi}\big)  \leq -r \right)
       &=& \P \left( \ttau^i  \geq  t_\ast-\gamma h^{-1} \sqrt{\psi} +\frac{r h^{-1}}{\sqrt{\psi}} \right)
\\
       &=& \P \left( \ttau^i\geq k_r\right),
\end{eqnarray*}
where
$$ 
       k_r:=t_\ast-\gamma h^{-1} \sqrt{\psi} +\frac{r h^{-1}}{\sqrt{\psi}}.
$$ 
We will show that for any $\eps,\sigma$ satisfying \eqref{eqn:moderate1} and \eqref{eqn:moderate2}
one has
\be \label{eqn:lim_ttaui}
  \P \left( \ttau^i\geq k_r\right) \to \P(\chi \leq -r)= \exp(-a_i e^{br})
\ee
for a random variable $\chi$ following the appropriate double exponential distribution.

\paragraph{Proof of the upper bound in \eqref{eqn:lim_ttaui}.}
{\it Step 1: Preliminaries.}

Let us introduce the auxiliary variables
$$
    L:=\psi^2,\qquad \ell:= B \psi,\qquad
    M:= \lfloor h^{-1} \psi^{-9/4} \rfloor,
$$
where $B>2/\amu$ (here $\amu$ is the exponential mixing constant defined in \eqref{eqn:mixingV}),
and the intervals
$$
   I_m:=k_r - m(L+\ell) +\ell + [0,L],\qquad m\in\{1,\ldots,M\}.
$$
Note that these intervals are disjoint and their union is a subset of $[0,k_r]$
(by the choice of $M$, $L$, and $\ell$), where we use \eqref{eqn:moderate2}.

Define the event of  exit of component $i$ in $I_m$ as follows:
\be \label{eqn:event_noexit}
    E_m^i :=  \left\lbrace\exists t\in I_m : \sigma \tV_t^i \ge \frac{t_\ast -t}{t_\ast}  \right\rbrace.
\ee

{\it Step 2: Decoupling.}
Let $\mix^i(\cdot)$ denote the mixing coefficient of the process $\tV^i$
(cf. the definition in Section~\ref{sec:decoupling}).

Since we chose $B>2/\amu$, we deduce from Lemma~\ref{lem:nonstr.4.2concrete} that
\begin{equation} \label{eqn:oh2}
    \mix^i(\ell)\leq e^{-\amu\ell} = e^{- \amu B \psi} = h^{\amu B} = o(h^2).
\end{equation}

Note that, since $\bigcup_{m=1}^M I_m \subseteq [0,k_r]$,
\begin{equation} \label{eqn:intermdecoup25}
   \P(\ttau^i \geq k_r)
   \leq  \P(  \bigcap_{m=1}^{M} (E_m^i)^c ) \leq \P(  \bigcap_{m=1}^{M-1} (E_m^i)^c ) \cdot
   \P((E_M^i)^c) + \mix^i(\ell),
\end{equation}
where we used the definition of $\mix^i$ and the fact that the intervals $I_m$ are separated
by a distance of at least $\ell$ and $E_m^i\in\sigma(\tV_t^i,t\in I_m)$.

Iterating this argument and using \eqref{eqn:oh2} gives
\begin{eqnarray}
   \P(\ttau^i \geq k_r)
   &\leq& \prod_{m=1}^M \P( (E_m^i)^c )  + M \mix^i(\ell)  \notag
\\
   &=& \prod_{m=1}^M \P( (E_m^i)^c )  + o(1) \notag
\\
   &=& \prod_{m=1}^M (1-\P( E_m^i ))  + o(1) \notag
\\  \label{eqn:contigg34}
   &\leq& \exp( -\sum_{m=1}^M \P( E_m^i ) )   + o(1).
\end{eqnarray}

{\it Step 3: Making the exit curve constant on $I_m$ and using the Pickands lemma.}
Recall that
$$
    \P( E_m^i ) =\P( \exists t\in I_m : \sigma \tV_t^i \geq \frac{t_\ast -t}{t_\ast} )
    = \P(  \exists t\in I_m : \tV_t^i \geq h \frac{t_\ast -t}{d} ).
$$
If we estimate $t$ in the term $h \frac{t_\ast -t}{d}$ by the left endpoint of $I_m$,
we get a lower bound for the probability. The left endpoint is given by
$$
    k_r - (L+\ell)m + \ell \geq k_r - (L+\ell)m
    = t_\ast - \gamma h^{-1}\sqrt{\psi} + \frac{r h^{-1}}{\sqrt{\psi}} - (L+\ell)m.
$$
Therefore, using stationarity of $\tV^i$,
\begin{eqnarray} \label{eqn:furtherexp27}
   \P( E_m^i )
   &\geq &
   \P\left( \exists t\in I_m : \tV_t^i \geq \frac{\gamma\sqrt{\psi}}{d}
   - \frac{r}{d\sqrt{\psi}} + \frac{h(L+\ell)m}{d}\right)
\\
   &=:&
   \P\left( \exists t\in I_m : \tV_t^i \geq Q_m\right)  \notag
\\
   &=& \P\left(  \exists t\in [0,L] : \tV_t^i \geq Q_m\right). \notag
\end{eqnarray}
The latter probability can be estimated via the Pickands lemma (Lemma~\ref{lem:pickands}), because
the right-hand side tends to zero, as we shall see immediately. The assumptions of Pickands
lemma are satisfied with $\alpha=1$ due to \eqref{eqn:lemma7directv1} in Lemma~\ref{lem:7directv1}.
Using also the definition $\gamma=\sqrt{2} d v$, for every fixed $\delta\in(0,1)$ we eventually obtain
for $\eps, \sigma$ small enough (depending on $\delta$)
\be \label{eqn:continuationpick1}
   \P( E_m^i ) \geq  (1-\delta) \frac{A_i}{\sqrt{\pi} } \, L \, \sqrt{\psi}
   \, \exp\left( - \frac{Q_m^2}{2v^2}\right),
\ee

By using that $\psi^{-1/2}\to 0$ and $h(L+\ell) m \leq 2 h L M \to 0$,
the expression $Q_m$ in \eqref{eqn:continuationpick1} can be further estimated as
\begin{eqnarray*}
 Q_m^2 &=& \left(\frac{\gamma\sqrt{\psi}}{d}
   - \frac{r}{d\sqrt{\psi}} + \frac{h(L+\ell)m}{d} \right)^2
\\
   &=&    \frac{\gamma^2\psi}{d^2} -
    \frac{2\gamma\sqrt{\psi}}{d} \left( \frac{r}{d\sqrt{\psi}} - \frac{h(L+\ell)m}{d}\right) +o(1),
\end{eqnarray*}
hence, using again $\gamma=\sqrt{2} d v$,
\begin{eqnarray*}
 \exp\left( - \frac{Q_m^2}{2v^2}\right)
 &=& (1+o(1)) \exp\left( - \frac{\gamma^2\psi}{2v^2d^2}
   + \frac{\gamma\sqrt{\psi}}{v^2d} \left( \frac{r}{d\sqrt{\psi}} - \frac{h(L+\ell)m}{d}\right)\right)
\\
&=& (1+o(1))\, h \exp\left(  \frac{\sqrt{2} r}{v d} - \frac{h \sqrt{2\psi} (L+\ell)m}{vd}\right);
\end{eqnarray*}
and we obtain eventually
\[
    \P( E_m^i )\geq
    (1-2\delta) \frac{A_i}{\sqrt{\pi}} \, L\, \sqrt{\psi}\, h
    \exp\left( \frac{\sqrt{2} r}{vd} - \frac{h\sqrt{2\psi} (L+\ell)m}{vd}\right).
\]
This shows that
\begin{eqnarray*}
    && \sum_{m=1}^M \P( E_m^i )
\\
    &\geq& (1-2\delta)\frac{A_i}{\sqrt{\pi} } L \sqrt{\psi} \, h\, e^{\frac{\sqrt{2} r}{vd}}
    \sum_{m=1}^M \exp\left(- \frac{h \sqrt{2\psi}(L+\ell)m}{vd} \right)
\\
    &\geq & (1-3\delta)\frac{A_i}{\sqrt{\pi} } L \sqrt{\psi} \, h \, e^{\frac{\sqrt{2} r}{vd}}
    \frac{1-\exp\left(- \frac{h \sqrt{2\psi}(L+\ell)(M+1)}{vd} \right)}
         {1-\exp\left(- \frac{h \sqrt{2\psi}(L+\ell)}{vd} \right)},
\end{eqnarray*}
where we used that we deal with a geometric sum (and the term related to $m=0$ is absorbed into one $\delta$).
We show in Step 4 that the numerator tends to one. The denominator tends to zero, so that
we get
\begin{eqnarray*}
    \sum_{m=1}^M \P( E_m^i )
    &\geq& (1-4\delta)\frac{A_i}{\sqrt{\pi} } L \sqrt{\psi} \, h \,
          e^{\frac{\sqrt{2} r}{vd}} \frac{1}{ \frac{h \sqrt{2\psi}(L+\ell)}{vd}}
\\
    &=& (1-4\delta)\frac{A_i}{\sqrt{\pi} } L \sqrt{\psi} \, h \,
      e^{\frac{\sqrt{2} r}{vd}} \frac{vd}{h \sqrt{2\psi}(L+\ell)}
\\
    &\geq & (1-5\delta)\frac{vdA_i}{\sqrt{2\pi}}\ e^{\frac{\sqrt{2} r}{vd}}.
\end{eqnarray*}

Plugging this back into \eqref{eqn:contigg34}, we obtain
\[
    \P(\ttau^i \geq k_r)
    \leq \exp\left( -(1-5\delta)\frac{vdA_i}{\sqrt{2\pi}} \  e^{\frac{\sqrt{2} r}{vd}} \right) +o(1).
\]
Letting now $\eps,\sigma\to 0$ and then $\delta\to 0$, one obtains that under \eqref{eqn:moderate2}
$$
    \limsup_{\eps,\sigma\to 0}
    \P\left( h\sqrt{\psi} \big( t_\ast-\ttau^i-\gamma h^{-1} \sqrt{\psi}\big) \leq -r \right)
    \leq \exp( -\frac{vd A_i}{\sqrt{2\pi}} \ e^{\frac{\sqrt{2} r}{vd}} ),
$$
as required for an upper bound. This also identifies the constants $a_i:=\frac{vdA_i}{\sqrt{2\pi}}$
and $b:=\sqrt{2}/(vd)$ in the proposition.

{\it Step 4: A postponed estimate from Step 3.} To complete Step 3, we still have to show
that $\exp(-  h \sqrt{2\psi}(L+\ell)(M+1)/(vd)) \to 0$. This follows from our choice of
parameters because
$$
   h \sqrt{\psi}(L+\ell)M \sim h \sqrt{\psi}\,  \psi^2 \, h^{-1} \psi^{-9/4}
   = \psi^{1/4} \to \infty.
$$

This finishes the proof of the upper bound in \eqref{eqn:lim_ttaui}. 

\paragraph{Proof of the lower bound in \eqref{eqn:lim_ttaui}.}

{\it Step 1: There is no early exit.}
Set $k_{-\infty} :=t_*-2\gamma h^{-1}\sqrt{\psi}$.
We prove that
\be \label{eqn:kminusinfty0}
  \P\left(\ttau^i\le k_{-\infty} \right)\to 0.
\ee
Let
$B_n:=[ k_{-\infty} -n-1, k_{-\infty}-n]$. Then, as $h\to 0$,
\begin{eqnarray*}
\P(\ttau^i\le k_{-\infty}) &=&
    \P\left(\exists t\in [0, k_{-\infty}]\, :\, \sigma \tV^i \ge \frac {t_\ast-t}{t_\ast} \right)
\\
    &\le& \sum_{n=0}^{\infty} \P\left( \sigma \max_{t\in B_n} \tV_t^i \ge \frac {2\gamma h^{-1}\sqrt{\psi} +n}{t_\ast} \right)
\\
   &=& \sum_{n=0}^{\infty} \P\left(  \max_{t\in B_n} \tV_t^i \ge  2v\sqrt{2\psi} + d^{-1} h n \right)
\\
      &\le& \sum_{n=0}^{\infty} \exp\left( -  (2v\sqrt{2\psi} + d^{-1} h n)^2/ (3v^2) \right)
\\
     &\le&  \exp\left( - 8 \psi/3 \right)   \sum_{n=0}^{\infty} \exp\left( -   4 \sqrt{2\psi} h n/ (3vd) \right)
\\
   &=&  h^{8/3}   \left( 1- \exp\left( -   4 \sqrt{2\psi} h / (3vd) \right)\right)^{-1} \to 0, 
\end{eqnarray*}
where we used (a weak version of) the Pickands lemma in the fourth step.

{\it Step 2: Evaluation of ``typical'' exit times.}

As in the proof of the upper bound, let us introduce the auxiliary variables
$$
    L:=\psi^2, \qquad \ell:= B \psi, \qquad
    M:= \lfloor 2 \gamma h^{-1} \sqrt{\psi} (L+\ell)^{-1} \rfloor,
$$
where $B>2/\amu$ (note that $M$ is of slightly smaller order than in the proof of the upper
bound, while $L$ and $\ell$ are the same as in that proof), and the intervals
$$
    J_m:=k_r - m(L+\ell) + (0,\ell),\
    I_m:=k_r - m(L+\ell) +\ell + [0,L],\ m\in\{1,\ldots,M\}.
$$
Note that these intervals are disjoint and their union covers (for small $h$) the interval
$[k_{-\infty},k_r]$ (by the choice of $M$, $L$, and $\ell$).
Indeed, one only has to check that
$k_r-M(L+\ell) \leq t_\ast- 2\gamma h^{-1} \sqrt{\psi}$
for small enough $h$, which is true.

We start with a trivial bound using \eqref{eqn:kminusinfty0}:
\begin{eqnarray} \notag
   \P(\ttau^i \ge k_r)&=& 1 - \P(\ttau^i\in [ k_{-\infty}, k_r])
     -  \P(\ttau^i <   k_{-\infty} )
\\  \label{eqn:lbexittimer25}
   &=&  \P(\ttau^i\not\in [ k_{-\infty}, k_r]) - o(1).
\end{eqnarray}
Next, define the events of exit of component $i$ in $I_m$ and $J_m$, respectively, as follows:
\begin{eqnarray} \label{eqn:defnemilower}
      E_m^i &:=&  \left\lbrace\exists t\in I_m :
        \sigma \tV_t^i \ge \frac{t_\ast -t}{t_\ast}  \right\rbrace,
\\
    N_m^i &:=&  \left\lbrace\exists t\in J_m :
        \sigma \tV_t^i \ge \frac{t_\ast -t}{t_\ast}  \right\rbrace. \notag
\end{eqnarray}

Note that
\begin{eqnarray}
    \P( \ttau^i\not \in [ k_{-\infty}, k_r] )
    &\geq&
    \P( \bigcap_{m=1}^M (E_m^i)^c \cap (N_m^i)^c) \notag
\\  \label{eqn:lower25splittingmain}
    &\geq&
    \P( \bigcap_{m=1}^M (E_m^i)^c ) - \P( \bigcup_{m=1}^M  N_m^i),
\end{eqnarray}
where the first step follows from the fact that the union
of $\bigcup_{m=1}^M (I_m\cup J_m)$ covers the interval
$[k_{-\infty},k_r]$,
as already observed above, while the second inequality is trivial.

We shall show that the first term in \eqref{eqn:lower25splittingmain} converges to the required
probability related to the double-exponential law (Steps 3-5), while the second term
in \eqref{eqn:lower25splittingmain} tends to zero (Step 6).

{\it Step 3: Decoupling.}

We estimate  $\P( \bigcap_{m=1}^M (E_m^i)^c )$ in the very same way as we did
in \eqref{eqn:intermdecoup25} and \eqref{eqn:contigg34}. In this way, for every
fixed $\delta\in(0,1)$, we obtain

\begin{eqnarray}
    \P( \bigcap_{m=1}^M (E_m^i)^c ) &\geq& \prod_{m=1}^M \P( (E_m^i)^c ) - o(1)  \notag
\\ \label{eqn:contigg34similar}
    &\geq& \exp( -\sum_{m=1}^M \P( E_m^i ) (1+\delta)) - o(1),
\end{eqnarray}
where we used that \eqref{eqn:oh2} is again valid (for the first step) and that
$\P( E_m^i) \to 0$, as we shall see presently, for the second step in
\eqref{eqn:contigg34similar}.

{\it Step 4: Making the exit curve constant on $I_m$ and using the Pickands lemma.}
Recall that
$$
    \P( E_m^i ) =\P( \exists t\in I_m : \sigma \tV_t^i \geq \frac{t_\ast -t}{t_\ast} )
    = \P\big( \exists t\in I_m : \tV_t^i \geq h \frac{t_\ast -t}{d} \big).
$$
Similarly to the proof of the upper bound (see \eqref{eqn:furtherexp27}),
one can estimate this expression -- this time from above -- by estimating $t$ by
the right endpoint of the interval -- as follows (using also stationarity in the
second step):
\begin{eqnarray}
   && \P( E_m^i )
    \leq
    \P\left( \exists t\in I_m : \tV_t^i \geq \frac{\gamma\sqrt{\psi}}{d}
    - \frac{r}{d\sqrt{\psi}} + \frac{h(L+\ell)(m+1)}{d}\right) \notag
\\
   &=:&
    \P\left(  \exists t\in I_m : \tV_t^i \geq Q_{m-1} \right) \notag
\\
    &=&
    \P\left(  \exists t\in [0,L] : \tV_t^i \geq Q_{m-1}  \right) \notag
\\  \label{eqn:continuationpick1llowerapp}
     &\leq & (1+\delta)\,
     \frac{A_i}{\sqrt{\pi} } \, L\, \sqrt{\psi}\, \exp\left( - \frac{Q_{m-1}^2}{2v^2} \right)
\\
    &\leq & (1+\delta)\,
     \frac{A_i}{\sqrt{\pi} } \, L\, \sqrt{\psi}\,
     \exp\left( -\frac{\gamma^2\psi}{2v^2d^2}
     +\frac{\gamma r}{v^2d^2} -\frac{\gamma\sqrt{\psi}h(L+\ell)(m+1)}{v^2d^2}\right) \notag
\\ 
       &=& (1+\delta)\,
        \frac{A_i}{\sqrt{\pi} } \, L\, \sqrt{\psi} \, h \, e^{\frac{\sqrt{2}r}{vd}}\,
        \exp\left(-\frac{\sqrt{2\psi}h(L+\ell)(m+1)}{vd}\right),
\end{eqnarray}
where we used the Pickands lemma (Lemma~\ref{lem:pickands})
in \eqref{eqn:continuationpick1llowerapp}, because the right-hand side tends to zero.
The assumptions of the Pickands lemma are satisfied due to \eqref{eqn:lemma7directv1}
in Lemma~\ref{lem:7directv1}. The estimate holds for $\eps, \sigma$ small enough
(depending on $\delta$).

{\it Step 5: Final computation for the main term in \eqref{eqn:lower25splittingmain}.}

Using the estimate from the last step, we obtain:
\begin{eqnarray*}
    && \sum_{m=1}^M \P( E_m^i )
\\
    &\leq&  (1+\delta)
    \frac{A_i}{\sqrt{\pi} } \, L\cdot \sqrt{\psi} \, h\, e^{\frac{\sqrt{2}r}{vd}}\sum_{m=1}^M
    \exp\left(- \frac{\sqrt{2\psi}h(L+\ell)(m+1)}{vd}\right)
\\
    &\leq&  (1+\delta)
    \frac{A_i}{\sqrt{\pi} } \, L\, \sqrt{\psi} \, h \, e^{\frac{\sqrt{2}r}{vd}}\sum_{m=0}^\infty
    \exp\left(- \frac{\sqrt{2\psi}h(L+\ell)m}{vd}\right)
\\
    &=&  (1+\delta)
    \frac{A_i}{\sqrt{\pi} } \, L\, \sqrt{\psi} \, h \, e^{\frac{\sqrt{2}r}{vd}}
    \frac{1}{1- \exp\left(- \frac{\sqrt{2\psi}h(L+\ell)}{vd}\right)}
\\
    &\leq&  (1+2\delta)
    \frac{A_i}{\sqrt{\pi} } \, L\, \sqrt{\psi} \, h \, e^{\frac{\sqrt{2}r}{vd}}
    \frac{1}{\frac{\sqrt{2\psi}h(L+\ell)}{vd}}
\\
    &\leq&  (1+2\delta)
\frac{ v dA_i}{\sqrt{2\pi} } \, e^{\frac{\sqrt{2}r}{vd}}.
\end{eqnarray*}

Plugging back the last term into \eqref{eqn:contigg34similar}, we obtain
\begin{equation} \label{eqn:contigg34similar2}
 \P( \bigcap_{m=1}^M (E_m^i)^c )
 \geq
 \exp\left( - (1+3\delta) \frac{vd A_i}{\sqrt{2\pi}} \, e^{\frac{\sqrt{2}r}{vd}}
   \right).
\end{equation}

{\it Step 6: The second term in \eqref{eqn:lower25splittingmain} tends to zero.}

Note that (replacing $t\in J_m$ by the right endpoint of that interval as in
the third step)
\begin{eqnarray*}
   \P( \bigcup_{m=1}^M  N_m^i)
   &\leq& \sum_{m=1}^M \P( N_m^i)
\\
   &\leq& \sum_{m=1}^M \P( \exists t \in J_m : \tV_t^i \geq h \frac{t_\ast-t}{d})
\\
    &\leq& \sum_{m=1}^M \P( \exists t \in J_m : \tV_t^i \geq h
    \frac{t_\ast-k_r+(m+1)(L+\ell)}{d})
\\
    &=& \sum_{m=1}^M \P( \exists t \in [0,\ell] : \tV_t^i
    \geq  v \sqrt{2\psi}-\frac{r}{ \sqrt{\psi}}+\frac{h(m+1)(L+\ell)}{d}),
\end{eqnarray*}
where we used stationarity in the last step. Again by the Pickands lemma
(Lemma~\ref{lem:pickands}), this term is upper bounded by
\begin{eqnarray*}
    && \sum_{m=1}^M c_i\, \ell\, \sqrt{\psi}\, \exp\left( - \frac{1}{2v^2}
       \left(  v \sqrt{2\psi}-\frac{r}{ \sqrt{\psi}}+\frac{h(m+1)(L+\ell)}{d}\right)^2 \right)
\\
     &\leq& \sum_{m=1}^M c_i\, \ell \,\sqrt{\psi}\, \exp\left( - \frac{1}{2v^2}
        \left(  2v^2\psi-2v\sqrt{2\psi}\left[\frac{r}{d\sqrt{\psi}}-\frac{h(m+1)(L+\ell)}{d}
        \right]\right) \right)
\\
     &\leq&  c_i \, \ell \,\sqrt{\psi}\, h\, e^{\frac{\sqrt{2}r}{vd}}
     \sum_{m=0}^\infty \exp\left( -\frac{\sqrt{2\psi}\,h \,m (L+\ell)}{v d} \right)
\\
     &=&  c_i\, \ell \,\sqrt{\psi}\, h \, e^{\frac{\sqrt{2}r}{vd}} \,
     \frac{1}{1- \exp\left( -\frac{\sqrt{2\psi}\, h \, (L+\ell)}{v d} \right)}
\\
     &\leq&  c_i \,\ell\, \sqrt{\psi}\, h \,e^{\frac{\sqrt{2}\, r}{vd}}
     \, \frac{2vd}{\sqrt{2\psi}\, h (L+\ell) }
\\
     &=&  \tilde c_i(r,d) \, \frac{\ell}{L+\ell} \to 0,
\end{eqnarray*}
as claimed.

By combining now this result with the estimates \eqref{eqn:lbexittimer25},
\eqref{eqn:lower25splittingmain}, \eqref{eqn:contigg34similar2} and letting
$\delta\to 0$, we obtain
\begin{eqnarray*}
     \liminf_{\eps,\sigma\to 0} \P\left( h \sqrt{\psi} \left(t_\ast-\ttau^i-\gamma h^{-1}
     \sqrt{\psi} \right) \leq -r \right)
     &=&
      \liminf_{\eps,\sigma\to 0} \P\left( \ttau^i \geq k_r \right)
\\
     &\geq& \exp( -\frac{vdA_i}{\sqrt{2\pi}}\,  e^{\frac{\sqrt{2} r}{vd}} ),
\end{eqnarray*}
as required for the lower bound.

This finishes the proof of the limit theorem for the exit times
in Proposition~\ref{prop:moderatelyslowpullingstat}.

\paragraph{Proof of the limit theorem for the break time $\ttau$.}
Would the break times for different chain elements $\ttau^i$ be fully independent,
the claim would be almost trivial by
\begin{eqnarray*}
   &&    \P\left( h\sqrt{\psi} \big( t_\ast-\ttau-\gamma h^{-1} \sqrt{\psi}\big)  \leq -r \right)
       = \P \left( \ttau\geq k_r\right)
       = \P \left( \min_{1\le i\le d}\ttau^i\geq k_r\right)
\\
        &=&  \prod_{i=1}^d  \P \left( \ttau^i\geq k_r\right) \to  \prod_{i=1}^d \exp( - a_i  e^{b r} )
         =  \exp\left( - \sum_{i=1}^d a_i  e^{br} \right) = \exp\left( - a_0  e^{br} \right).
\end{eqnarray*}
Since they are only {\it asymptotically} independent, the result remains the same but the justification
requires more technical efforts. Here we only trace the proof of the upper and lower bounds.

Imitating the previous proof of the upper bound, one obtains
\begin{eqnarray*}
    \P(\ttau\ge k_r) &\leq& \P\left( \bigcap_{i=1}^d\bigcap_{m=1}^M (E_m^i)^c \right)
  \leq \prod_{m=1}^{M}  \P\left( \bigcap_{i=1}^d (E_m^i)^c \right) +o(1)
\\
  &\leq& \exp\left( - \sum_{m=1}^{M}  \P\left( \bigcup_{i=1}^d  E_m^i \right) \right)+o(1),
\end{eqnarray*}
where $E_m^i$ are defined in \eqref{eqn:event_noexit}. By the inclusion-exclusion formula,
\[
   \P\left( \bigcup_{i=1}^d  E_m^i \right) \ge  \sum_{i=1}^d \P\left(  E_m^i \right)
   - \sum_{1\le i_1<i_2\le d} \P\left(  E_m^{i_1} \cap E_m^{i_2} \right).
\]
In fact, the second sum is negligible with respect to the first one because, using notation
from \eqref{eqn:furtherexp27}, for simple exit probabilities we have
\[
  \P\left(  E_m^i \right) \geq \exp\left(- \frac{Q_m^2(1+o(1))}{2v^2} \right)
\]
while, by the standard arguments on Gaussian extrema, for double exit probabilities we have for $i_1\neq i_2$
\[
  \P\left( E_m^{i_1} \cap E_m^{i_2}  \right) \leq \exp\left(-\frac{\eta\, Q_m^2 (1+o(1))}{2v^2}\right)
\]
with $\eta>1$ depending on $\max_{0\le t_1,t_2\le L} \covcoef(\tV^{i_1}_{t_1}, \tV^{i_2}_{t_2})<1$,
cf.\ the argument leading to \eqref{eqn:admit7.3}. Thus we obtain
\begin{eqnarray*}
  \P(\ttau\ge k_r) &\leq&
  \exp\left( -  \sum_{i=1}^d \sum_{m=1}^{M}  \P\left( E_m^i \right) (1+o(1)) \right)+o(1)
\\
 &\leq&     \exp\left( - \sum_{i=1}^d   \frac{vdA_i}{\sqrt{2\pi}} e^{\frac{\sqrt{2}r}{vd}}  (1+o(1)) \right)+o(1)
\\
  &=&     \exp\left( - a_0 e^{br} (1+o(1)) \right)+o(1),
\end{eqnarray*}
which is the desired upper bound.

On the other hand, imitating the previous proof of the lower bound, for every fixed
$\delta>0$, one obtains
\begin{eqnarray*}
  \P(\ttau\ge k_r) &\geq& \P\left( \bigcap_{i=1}^d \bigcap_{m=1}^M (E_m^i)^c \right)
   - \P\left( \bigcup_{i=1}^d \bigcup_{m=1}^M  N_m^i\right) -o(1)
\\
   &\geq& \exp\left( -  \sum_{i=1}^d \sum_{m=1}^M \P(E_m^i )(1+\delta)\right)
   - \sum_{i=1}^d \sum_{m=1}^M  \P(N_m^i) -o(1)
\\
      &\geq& \exp\left( -  \sum_{i=1}^d (1+3\delta)  \frac{vdA_i}{\sqrt{2\pi}} e^{\frac{\sqrt{2}r}{vd}} (1+\delta)\right) -o(1)
\\
      &=& \exp\left( -  (1+3\delta) a_0 e^{br} (1+\delta)\right) -o(1),
\end{eqnarray*}
where $E_m^i$ and $N_m^i$ are defined in \eqref{eqn:defnemilower}.
Letting $\delta\to 0$ proves the lower bound.

\paragraph{Proof of the limit theorem for the exit position.}
We have to prove that under conditions \eqref{eqn:moderate1} and \eqref{eqn:moderate2}
$\P(\ttau=\ttau^i)$ tends to the limiting probabilities $p_i$ in the theorem, which can be expressed as $p_i=A_i/(2d)$.
Since $\sum_{i=1}^{d} p_i=1$ it is sufficient to prove the upper bound
\[
 \limsup_{\eps,\sigma\to 0} \P(\ttau=\ttau^i) \le  p_i.
\]
Let us fix a large $r>0$, let $L:=B \psi$ with $B>2/\mu$
(here, $\mu$ is the exponential mixing constant defined in \eqref{eqn:mixingV}),
and $M:=\lceil\frac{h^{-1} r}{L\sqrt{\psi}}\rceil$.
Introduce the sequence of intervals
\[
   I_m:=k_0+mL+[0,L], \qquad -M\le m\le M.
\]
Then we have the bounds for the corresponding leftmost and rightmost points of $I_{-M}$ and $I_M$, respectively,
\begin{eqnarray*}
   k_0-ML &\leq& k_0- \frac{h^{-1} r}{\sqrt{\psi}} =k_{-r};
\\
   k_0+(M+1)L &\geq& k_0 + \frac{h^{-1} r}{\sqrt{\psi}} =k_r.
\end{eqnarray*}
Therefore,
\[
  \P(\ttau=\ttau^i) \le \P(\ttau\le k_{-r}) + \P(\ttau\ge k_{r})
  +\sum_{m=-M}^{M} \P(\ttau=\ttau^i\in I_m).
\]
By the limit theorem for the exit time $\ttau$, cf. \eqref{eqn:moderateexittimescaling}, we have
\begin{eqnarray}
   \lim_{\eps,\sigma\to 0}  \P(\ttau\le k_{-r})  &=& 1-\exp\left(-a_0 e^{-br}\right), \label{eqn:2111a}
\\
   \lim_{\eps,\sigma\to 0}  \P(\ttau\ge k_{r})  &=& \exp\left(-a_0 e^{br}\right),\label{eqn:2111b}
\end{eqnarray}
with $a_0=\tfrac{2vd^2}{\sqrt{2\pi}}$, $b=\tfrac{\sqrt{2}}{vd}$.

We also note that
\begin{eqnarray*}
 \{ \ttau^i\in I_m\} &\subseteq&
   \left\{ \exists t\in I_m: \sigma \tV_t^i\ge  \frac{t_\ast-t}{t_\ast} \right\}
 \subseteq
   \left\{ \max_{t\in I_m} \tV_t^i \ge d^{-1}(\gamma\sqrt{\psi}-\frac{r}{\sqrt{\psi}}-Lh )  \right\}
\\
 &=& \left\{ \max_{t\in I_m} \tV_t^i \ge  v \sqrt{2\psi}- \frac{r}{d\sqrt{\psi}} -Lh \right\}.
\end{eqnarray*}

Now we may use the mixing property, i.e.\ Lemma~\ref{lem:nonstr.4.2concrete}, and obtain
\begin{eqnarray*}
  &&  \sum_{m=-M}^{M} \P(\ttau \ge k_0 + (m-1)L; \ttau^i\in I_m)
\\
  &\leq&  \sum_{m=-M}^{M} \P(\ttau \ge k_0 + (m-1)L;
    \max_{t\in I_m} \tV_t^i \ge  v \sqrt{2\psi}- \frac{r}{d\sqrt{\psi}} -Lh )
\\
   &\leq& \sum_{m=-M}^{M} [ \P(\ttau \ge k_0 + (m-1)L) \
     \P(  \max_{t\in I_m} \tV_t^i \ge v \sqrt{2\psi}- \frac{r}{d\sqrt{\psi}}-Lh )+e^{-\amu L}]
\\
   &=&  (2M+1)e^{-\amu L} + \sum_{m=-M}^{M} \P(\ttau \ge k_0 + (m-1)L)\
   \P( \max_{t\in I_m} \tV_t^i \ge  v \sqrt{2\psi}- \frac{r}{d\sqrt{\psi}} -Lh).
\end{eqnarray*}
Notice that $\amu B>1$ yields
\[
  (2M+1) e^{-\amu L} \sim  \frac{2 h^{-1} r}{L\sqrt{\psi}} \ e^{-\amu B\psi} = h^{\amu B-1} \, \frac{2r}{B\psi^{3/2}} \to 0,
  \qquad \textrm{as } h\to 0.
\]
It remains to evaluate the probabilities in the sum. Let $\rho_m:=m \, L \, \sqrt{\psi}\, h$.
Then $k_{\rho_m}= k_0+mL$ and by using again the limit theorem for $\ttau$ we have
\[
    \P(\ttau \ge k_0 + (m-1)L) = \P(\ttau\ge k_{\rho_m}-L) = \exp\left( -a_0 e^{b\rho_m}\right) (1+o(1)),
\]
where we used that $|\rho_m|\leq r$. On the other hand, by the Pickands lemma,
\begin{eqnarray*}
  && \P\left( \max_{t\in I_m} \tV_t^i \ge  v \sqrt{2\psi}- \frac{\rho_m}{d\sqrt{\psi}}-Lh \right)
\\
   &=& \frac{A_i L}{\sqrt{2\pi}} \ \sqrt{2\psi}\ \exp( -\psi + b\rho_m)\, (1+o(1))
\\
   &=& \frac{A_i L}{\sqrt{2\pi}} \ \sqrt{2\psi}\, h \, \exp( b\rho_m) \, (1+o(1))
\\
   &=& \frac{A_i}{\sqrt{\pi}}\  \exp(b\rho_m)\, (\rho_m-\rho_{m-1}) \,  (1+o(1)).
\end{eqnarray*}
We conclude that
\begin{eqnarray*}
 && \sum_{m=-M}^{M} \P(\ttau=\ttau^i\in I_m)
\\
  &\leq&  \frac{A_i}{\sqrt{\pi}}
  \sum_{m=-M}^{M}  \exp\left( -a_0 e^{b\rho_m}\right) \ \exp(b\rho_m)\, (\rho_m-\rho_{m-1}) +o(1),
\end{eqnarray*}
and recognize the Riemann sum of the integral
\begin{eqnarray*}
  \int_{-r}^{r} \exp\left( -a_0 e^{b\rho}\right) \, \exp(b\rho)\, \dd\rho
  &=&
   b^{-1} \int_{e^{-br}}^{e^{br}} \exp\left( -a_0 z \right) \, \dd z
\\
   &=&
   (ba_0)^{-1} \left[ \exp\left( -a_0 e^{-br} \right) -  \exp\left( -a_0 e^{br} \right) \right].
\end{eqnarray*}
Therefore, taking the limit in $\eps,\sigma\to 0$ with fixed $r$ yields (when also using
\eqref{eqn:2111a} and \eqref{eqn:2111b})
\begin{eqnarray*}
  && \limsup_{\eps,\sigma\to 0} \P(\ttau=\ttau^i) \le
   1- \exp\left(-a_0 e^{-br}\right) + \exp\left(-a_0 e^{br}\right)
\\
  &&  +
  \frac{A_i}{\sqrt{\pi}}\ (ba_0)^{-1}
  \left[ \exp\left( -a_0 e^{-br} \right) -  \exp\left( -a_0 e^{br} \right) \right].
\end{eqnarray*}
Finally, letting $r\to +\infty$, we obtain
\[
  \limsup_{\eps,\sigma\to 0} \P(\ttau=\ttau^i)
  \le   \frac{A_i}{\sqrt{\pi}}\ (ba_0)^{-1} = \frac{A_i}{2d} = p_i,
 \]
as required.

\subsection{Proof of Theorem \ref{thm:intermediateslowpulling}}

Let $\zeta:= t_\ast/2 = \frac{d}{2\eps}$.

\subsubsection{Limit theorem for $\tau^i$}

\paragraph{Upper bound for the exit time.}
Using the sufficient condition for the exit \eqref{eqn:exitsufficientshort},
we obtain that for every $F>0$ and every real $r$
\begin{eqnarray*}
   \P(\tau^i>k_r) &\leq&
   \P\left(\sigma \tV^i_t < 1 - \frac{\eps t}{d} + \eps K^i +\sigma\,\Xi\, e^{-\amu t},\ 0\le t\le k_r\right)
\\
    &\leq&
   \P\left(\sigma \tV^i_t < 1 - \frac{\eps t}{d} + \eps K^i + \sigma\, F\, e^{-\amu t},\ 0\le t\le k_r \right) + \P(\Xi>F)
\\
   &\leq&
   \P\left(\sigma \tV^i_t < 1 - \frac{\eps t}{d} + \eps K^i + \sigma\, F\, e^{-\amu \zeta},\ \zeta \le t\le k_r \right) + \P(\Xi>F)
\\
    &=&
   \P\left(\sigma \tV^i_t < 1 - \frac{\eps (t-u)}{d}, \zeta \le t\le k_r\right) + \P(\Xi>F),
\end{eqnarray*}
where
\be \label{eqn:u}
  u:= \frac{d}{\eps}\, (\eps K^i +\sigma\, F\, e^{-\amu \zeta})
  = d\,K^i +h^{-1} d\, F\,  e^{-\amu \zeta} \ll \left(\frac{h^{-1}}{\psi}\right),
\ee
because $e^{\amu \zeta} \gg \ln(\eps^{-1}) \ge \ln(\sigma \eps^{-1})  = \psi$.
Using the stationarity of $\tV^i$, we continue the evaluation with
\begin{eqnarray}
 &&   \P\left(\sigma \tV^i_t < 1 - \frac{\eps (t-u)}{d},\ \zeta \le t\le k_r\right) \label{eqn:2111cap1}
\\
    &=&
     \P\left(\sigma \tV^i_{s+u} < 1 - \frac{\eps s}{d},\  \zeta -u\le s\le k_r-u\right)\notag
\\
    &\leq&  \P\left(\sigma \tV^i_{s+u} < 1 - \frac{\eps s}{d},\  \zeta \le s\le k_r-u\right)\notag
\\
   &=&  \P\left(\sigma \tV^i_{s} < 1 - \frac{\eps s}{d},\ \zeta \le s\le k_r-u\right)\notag
\\
   &\le&  \P\left(\sigma \tV^i_{s} < 1 - \frac{\eps s}{d},\ 0 \le s\le k_r-u\right)
   +  \P\left(\exists s\in [0,\zeta]: \sigma \tV^i_{s} \ge 1 - \frac{\eps s}{d}\right) \notag
\\
   &=& \P(\ttau^i>k_r-u) + \P(\ttau^i\le \zeta). \label{eqn:2111cap2}
\end{eqnarray}
The bound in \eqref{eqn:u} shows that for every $\delta>0$ the inequality
$k_r-u> k_r - \delta h^{-1} \psi^{-1/2} = k_{r-\delta}$ is true eventually.

On the other hand, for every real $\rho$ we also have
eventually
\be \label{eqn:zetasmall}
   \zeta < k_\rho.
\ee

Therefore, we obtain by using \eqref{eqn:moderateexittimescaling_i}
in Proposition \ref{prop:moderatelyslowpullingstat},
\begin{eqnarray*}
   \limsup_{\eps,\sigma\to 0}  \P(\tau^i>k_r) &\le& \lim_{\eps,\sigma\to 0}  \P(\ttau^i>k_{r-\delta}) + \lim_{\eps,\sigma\to 0} \P(\ttau^i\le k_\rho)
   + \P(\Xi>F)
\\
   &=& e^{-a_ie^{b(r-\delta)}} + [1- e^{-a_ie^{b\rho}}]  + \P(\Xi>F).
\end{eqnarray*}
By letting $\delta\to 0$, $\rho\to -\infty$, and $F\to+\infty$ we obtain the desired
\[
   \limsup_{\eps,\sigma\to 0}  \P(\tau^i>k_r) \le  e^{-a_ie^{br}}.
\]

\paragraph{Lower bound for the exit time.}

Using the necessary condition for the exit \eqref{eqn:exitnecessaryshort} we obtain
that for every $F>0$ and every real $r$
\begin{eqnarray*}
   \P(\tau^i\ge k_r) &\geq&
   \P\left(\sigma \tV^i_t < 1 - \frac{\eps t}{d} - \eps K^i - \sigma\, \Xi\, e^{-\amu t},\  0\le t\le k_r\right)
\\
    &\geq&
   \P\left(\sigma \tV^i_t < 1 - \frac{\eps t}{d} - \eps K^i - \sigma\, F\, e^{-\amu t},\ 0\le t\le k_r \right) - \P(\Xi>F)
\\
    &\geq& \P(B\cap D) - \P(\Xi>F) \ge \P(B)-\P(D^c) - \P(\Xi>F),
\end{eqnarray*}
where
\begin{eqnarray*}
  B &:=& \left\{  \sigma \tV^i_t < 1 - \frac{\eps t}{d} - \eps K^i - \sigma\, F\, e^{-\amu \zeta},\ 0\le t\le k_r \right\};
\\
  D &:=& \left\{  \sigma \tV^i_t < 1 - \frac{\eps t}{d} - \eps K^i - \sigma\, F,\ 0\le t\le \zeta \right\}.
\end{eqnarray*}
For the main part we have
\begin{eqnarray*}
  \P(B) &=& \P\left(\sigma \tV^i_t < 1 - \frac{\eps (t+u)}{d}, 0\le t\le k_r\right)
\\
  &=& \P\left(\sigma \tV^i_{s-u} < 1 - \frac{\eps s}{d}, u \le s\le k_r+u\right)
\\
  &\ge&  \P\left(\sigma \tV^i_{s-u} < 1 - \frac{\eps s}{d}, 0 \le s\le k_r+u\right),
\end{eqnarray*}
where $u$ is the same as in \eqref{eqn:u}.
By using stationarity of $\tV^i$, we continue the evaluation with
\[
    \P(B) \geq  \P\left(\sigma \tV^i_{s} < 1 - \frac{\eps s}{d}, 0\le s\le k_r+u\right)
= \P(\ttau^i>k_r+u).
\]

The bound in \eqref{eqn:u} shows that for every $\delta>0$ the inequality
$k_r+u<k_r +\delta h^{-1} \psi^{-1/2}= k_{r+\delta}$ is true eventually. Therefore, we obtain by using \eqref{eqn:moderateexittimescaling_i}
in Proposition \ref{prop:moderatelyslowpullingstat},
\[
   \liminf_{\eps,\sigma\to 0}  \P(B) \ge \lim_{\eps,\sigma\to 0}  \P(\ttau^i>k_{r+\delta})= e^{-a_ie^{b(r+\delta)}}.
\]

For the remainder term, we have
\begin{eqnarray*}
  \P(D) &=& \P\left(\sigma \tV^i_t < 1 - \frac{\eps (t+u_D)}{d}, 0\le t\le \zeta\right)
\\
  &=& \P\left(\sigma \tV^i_{s-u_D} < 1 - \frac{\eps s}{d}, u_D \le s\le  \zeta +u_D\right)
\\
  &\ge&  \P\left(\sigma \tV^i_{s-u_D} < 1 - \frac{\eps s}{d},  0 \le s\le  \zeta +u_D \right),
\end{eqnarray*}
where
\be \label{eqn:uD}
   u_D:=d\,K^i +h^{-1} d\, F = O\left(h^{-1}\right)= o\left(\eps^{-1}\right).
\ee

By using stationarity of $\tV^i$, we continue the evaluation with
\[
    \P(D) \geq  \P\left(\sigma \tV^i_{s} < 1 - \frac{\eps s}{d}, 0\le s\le \zeta +u_D \right)
= \P(\ttau^i> \zeta +u_D).
\]
Since for every real $\rho$ we have eventually $\zeta +u_D= \zeta(1+o(1))<k_\rho$, it follows
from \eqref{eqn:moderateexittimescaling_i} that
\[
 \liminf_{\eps,\sigma\to 0} \P(D)\geq \lim_{\eps,\sigma\to 0} \P(\ttau^i> k_\rho) = e^{-a_ie^{b\rho}}.
\]
Equivalently, we have $\limsup_{\eps,\sigma\to 0} \P(D^c)\leq 1- e^{-a_ie^{b\rho}}$. Therefore,
\[
   \liminf_{\eps,\sigma\to 0}  \P(\tau^i>k_r) \geq e^{-a_ie^{b(r+\delta)}}
     + [ 1- e^{-a_ie^{b\rho}} ] -  \P(\Xi>F).
\]
By letting $\delta\searrow 0$, $\rho\to-\infty$, and $F\to+\infty$, we obtain the desired
\[
   \liminf  \P(\tau^i>k_r) \ge  e^{-a_i e^{br}}.
\]

\subsubsection{Limit theorem for $\tau$}

\paragraph{Upper bound for the exit time.}
The proof of the upper bound is essentially the same as that for $\tau^i$. Using sufficient condition
for the exit \eqref{eqn:exitsufficientshort} we see that for every $F>0$ and every real $r$
\begin{eqnarray*}
     &&  \P(\tau>k_r)
\\
     &\leq&  \P\left(\sigma \tV^i_t < 1 - \frac{\eps t}{d} + \eps K^i +\sigma\Xi e^{-\amu t},\
         0\le t\le k_r, 1\le i\le d \right)
\\
     &\leq&  \P\left(\sigma \tV^i_t < 1 - \frac{\eps t}{d} + \eps K^i + \sigma \, F \, e^{-\amu t}, 0\le t\le k_r,\
         1\le i\le d  \right) + \P(\Xi>F)
\\
    &\leq& \P\left(\sigma \tV^i_t < 1 - \frac{\eps t}{d} + \eps K^i + \sigma \, F \, e^{-\amu \zeta}, \
         \zeta \le t\le k_r, 1\le i\le d  \right) + \P(\Xi>F)
\\
    &\leq& \P\left(\sigma  \max_{1\le i\le d} \tV^i_t < 1 - \frac{\eps (t-\ustar)}{d}, \
         \zeta \le t\le k_r   \right) + \P(\Xi>F),
\end{eqnarray*}
where
\begin{eqnarray} \label{eqn:ustar}
    \ustar&:=& \frac{d}{\eps}\, (\eps  \max_{1\le i\le d} K^i +\sigma\, F  \, e^{-\amu \zeta} )
\\ \label{eqn:ustar_bound}
    &=&   d \, \max_{1\le i\le d} K^i + h^{-1} d\, F \, e^{-\amu \zeta}
    \ll \frac{h^{-1}}{\sqrt{\psi}}.
\end{eqnarray}
Using the stationarity of $\tV^i$, we continue obtain as in \eqref{eqn:2111cap1}-\eqref{eqn:2111cap2} that
\begin{eqnarray*}
      &&  \P\left(\sigma \max_{1\le i\le d} \tV^i_t < 1 - \frac{\eps (t-\ustar)}{d}, \
          \zeta \le t\le k_r\right)
\\
			&\leq& \P(\ttau>k_r-\ustar) + \P(\ttau\le \zeta).
\end{eqnarray*}
The bound in \eqref{eqn:ustar_bound} shows that for every $\delta>0$ one has eventually
$k_r-\ustar>k_r - \delta h^{-1} \psi^{-1/2}= k_{r-\delta}$.
Using \eqref{eqn:zetasmall} and \eqref{eqn:moderateexittimescaling}
in Proposition \ref{prop:moderatelyslowpullingstat}, we obtain
\begin{eqnarray*}
   && \limsup_{\eps,\sigma\to 0}  \P(\tau >k_r) \leq  \lim_{\eps,\sigma\to 0}  \P(\ttau >k_{r-\delta})
   +  \lim_{\eps,\sigma\to 0}  \P(\ttau \le k_\rho) + \P(\Xi>F)
\\
   &=& e^{-a_0 e^{b(r-\delta)}} + [1-e^{-a_0 e^{b\rho}}] + \P(\Xi>F).
\end{eqnarray*}
By letting $\delta\searrow 0$, $\rho\to-\infty$, and $F\to+\infty$ we obtain the desired
\[
   \limsup_{\eps,\sigma\to 0}  \P(\tau > k_r) \le  e^{-a_0e^{br}}.
\]

\paragraph{Lower bound for the exit time.}
The proof of the lower bound is essentially the same as that for $\tau^i$. Using the necessary condition
for the exit \eqref{eqn:exitnecessaryshort} we see that for every $F>0$ and every real $r$
\begin{eqnarray*}
   \P(\tau\ge k_r) &\geq&
   \P\left(\sigma \tV^i_t < 1 - \frac{\eps t}{d} - \eps K^i - \sigma\, \Xi\, e^{-\amu t}, \
    0\le t\le k_r, 1\le i\le d\right)
\\
    &\geq& \P\left(\sigma \tV^i_t < 1 - \frac{\eps t}{d} - \eps K^i - \sigma\, F \, e^{-\amu t},\
    0\le t\le k_r, 1\le i\le d \right) - \P(\Xi>F)
\\
    &\geq& \P(B\cap D) - \P(\Xi>F) \ge \P(B)-\P(D^c) - \P(\Xi>F),
\end{eqnarray*}
where this time
\begin{eqnarray*}
  B &:=& \left\{  \sigma \tV^i_t < 1 - \frac{\eps t}{d} - \eps K^i - \sigma\, F\, e^{-\amu \zeta},\
       0\le t\le k_r, 1\le i\le d \right\};
\\
  D &:=& \left\{  \sigma \tV^i_t < 1 - \frac{\eps t}{d} - \eps K^i - \sigma\, F,\
      0\le t\le \zeta,  1\le i\le d \right\}.
\end{eqnarray*}
For the main part we have
\begin{eqnarray*}
  \P(B) &\ge& \P\left(\sigma \max_{1\le i\le d} \tV^i_t < 1 - \frac{\eps (t+\ustar)}{d}, \
       0\le t\le k_r \right)
\\
    &=& \P\left(\sigma \max_{1\le i\le d} \tV^i_{s-\ustar} < 1 - \frac{\eps s}{d}, \
       \ustar\le s \le k_r+\ustar \right)
\\
     &\geq& \P\left(\sigma \max_{1\le i\le d} \tV^i_{s-\ustar} < 1 - \frac{\eps s}{d}, \
       0 \le s \le k_r+\ustar \right),
\end{eqnarray*}
where $\ustar$ is the same as in \eqref{eqn:ustar}.

Using the stationarity of $\tV^i$, we continue the evaluation with
\[
  \P(B) \geq \P\left(\sigma \max_{1\le i\le d} \tV^i_{s} < 1 - \frac{\eps s}{d}, \
       0 \le s \le k_r+\ustar \right)
  = \P(\ttau > k_r+\ustar)
\]
Using the bound in \eqref{eqn:ustar_bound}, for every $\delta>0$ one has eventually
$k_r+\ustar< k_r + \delta h^{-1} \psi^{-1/2} = k_{r+\delta}$. Therefore, we obtain by using \eqref{eqn:moderateexittimescaling}
in Proposition \ref{prop:moderatelyslowpullingstat},
\[
   \liminf_{\eps,\sigma\to 0}  \P(B) \geq \lim_{\eps,\sigma\to 0}  \P(\ttau > k_{r+\delta}) = e^{-a_0 e^{b(r+\delta)}}.
\]

For the remainder term, we have
\begin{eqnarray*}
  \P(D) &\geq&  \P( \sigma  \max_{1\le i\le d} \tV^i_t < 1 - \frac{\eps (t+u_D)}{d},\
      0\le t\le \zeta )
\\
    &\geq&  \P( \sigma  \max_{1\le i\le d} \tV^i_{s-u_D} < 1 - \frac{\eps s}{d},\
      u_D\le s\le \zeta+u_D )
\\
    &\geq&  \P( \sigma  \max_{1\le i\le d} \tV^i_{s-u_D} < 1 - \frac{\eps s}{d},\
      0\le s\le \zeta+u_D ),
\end{eqnarray*}
where $u_D$ is as defined in \eqref{eqn:uD}. By stationarity,
\[
   \P(D) \geq \P( \sigma  \max_{1\le i\le d} \tV^i_{s} < 1 - \frac{\eps s}{d}\
      0\le s\le \zeta+u_D )
   = \P (\ttau > \zeta+u_D).
\]
Since for every $\rho$ we have eventually $\zeta +u_D = \zeta(1+o(1))<k_\rho$, it follows
from \eqref{eqn:moderateexittimescaling} that
\[
 \liminf_{\eps,\sigma\to 0} \P(D)\geq  \lim_{\eps,\sigma\to 0} \P(\ttau> k_\rho) =  e^{-a_0 e^{b\rho}} .
\]
In other words, we have $ \limsup_{\eps,\sigma\to 0} \P(D^c) \leq 1- e^{-a_0 e^{b\rho}}$. Therefore,
\[
   \liminf_{\eps,\sigma\to 0}  \P(\tau>k_r) \geq e^{-a_0 e^{b(r+\delta)}} - [1- e^{-a_0 e^{b\rho}}] -  \P(\Xi>F).
\]
By letting $\delta\searrow 0$, $\rho\to-\infty$, and $F\to+\infty$, we obtain the desired
\[
   \liminf_{\eps,\sigma\to 0}  \P(\tau>k_r) \ge  e^{-a_0 e^{br}}.
\]

\subsubsection{Limit theorem for the break position}
We have to prove that under conditions \eqref{eqn:moderate1} and \eqref{eqn:moderate2}
\[
   \lim_{\eps,\sigma\to 0} \P(\tau=\tau^i)= p_i:=
    \begin{cases}
    \frac{1}{d-1} & i\in\{2,\ldots,d-1\};\\
    \frac{1}{2(d-1)} & i\in\{1,d\}.
    \end{cases}
\]
Since $\sum_{i=1}^{d} p_i=1$ it is sufficient to prove the lower bound
\be \label{eqn:liminf_face}
   \liminf_{\eps,\sigma\to 0} \P(\tau=\tau^i) \ge  p_i.
\ee
Let us fix a large $r>0$ and $F>0$. Recall the notation $\zeta:=t_\ast/2=\tfrac{d}{2\eps}$. Using the necessary
and the sufficient conditions for exit \eqref{eqn:exitnecessaryshort} and \eqref{eqn:exitsufficientshort}, we have
\begin{eqnarray*}
   &&     \P(\tau=\tau^i) \geq  \P(\zeta \le \tau=\tau^i \le k_r) \
\\
   &\geq&   \P\Big(\tau\ge \zeta \textrm{ and } \exists t\in[\zeta,k_r] : \sigma \tV^i_t \ge 1 - \frac{\eps t}{d} + \eps K^i +\sigma\, \Xi\, e^{-\amu t}
\\
   && \qquad \textrm{ and } \sigma \tV^j_s < 1 - \frac{\eps s}{d} - \eps K^j - \sigma\, \Xi\, e^{-\amu s} , \
    \zeta \le s\le t, j\not=i \Big)
\\
    &\geq&   \P\Big(\exists t\in[\zeta,k_r] : \sigma \tV^i_t \ge 1 - \frac{\eps t}{d} + \eps K^i +\sigma\, F\, e^{-\amu \zeta}
\\
   && \qquad \textrm{ and } \sigma \tV^j_s < 1 - \frac{\eps s}{d} - \eps K^j - \sigma\, F\, e^{-\amu \zeta} , \
    \zeta \le s\le t, j\not=i \Big)
\\
   &&     - \P(\tau\le \zeta)-\P(\Xi>F)
\\
    &\geq&   \P\Big(\exists t\in[\zeta,k_r] : \sigma \tV^i_t \ge 1 - \frac{\eps (t-\ustar)}{d}
\\
   && \qquad \textrm{ and } \sigma \tV^j_s < 1 - \frac{\eps (s+\ustar)}{d} , \
    \zeta \le s\le t, j\not=i \Big)
\\
    &&     - \P(\tau\le \zeta)-\P(\Xi>F),
\end{eqnarray*}
with $\ustar$ defined in \eqref{eqn:ustar}. Furthermore,
\begin{eqnarray*}
   && \P\Big(\exists t\in[\zeta,k_r] : \sigma \tV^i_t \ge 1 - \frac{\eps (t-\ustar)}{d}
\\
   && \qquad \textrm{ and } \sigma \tV^j_s < 1 - \frac{\eps (s+\ustar)}{d}, \
    \zeta \le s\le t, j\not=i \Big)
\\
   &=&   \P\Big( \exists t'\in[\zeta+\ustar,k_r+\ustar] : \sigma \tV^i_{t'-\ustar} \ge 1 - \frac{\eps (t'-2\ustar)}{d}
\\
   && \qquad \textrm{ and } \sigma \max_{j\not=i} \tV^j_{s'-\ustar} < 1 - \frac{\eps s'}{d},
       \zeta+\ustar \le s'\le t' \Big)                                                                    \qquad\qquad  (t'=t+\ustar,s'=s+\ustar)
\\
   &\geq& \P\Big( \exists t'\in[\zeta+\ustar,k_r] : \sigma \tV^i_{t'-\ustar} \ge 1 - \frac{\eps (t'-2\ustar)}{d}
\\
   && \qquad \textrm{ and } \sigma \max_{j\not=i} \tV^j_{s'-\ustar} < 1 - \frac{\eps s'}{d},
       0 \le s'\le t' \Big).
\end{eqnarray*}
Using the stationarity of $\tV$, the last term equals
\begin{eqnarray*}
	&&\P\Big( \exists t\in[\zeta+\ustar,k_r] : \sigma \tV^i_{t} \ge 1 - \frac{\eps (t-2\ustar)}{d}
\\
   && \qquad \textrm{ and } \sigma \max_{j\not=i} \tV^j_{s} < 1 - \frac{\eps s}{d},
       0 \le s\le t \Big)
\\
   &\geq& \P\Big( \exists t\in[\zeta+\ustar,k_r] : \sigma_{*i} \tV^i_{t} \ge 1 - \frac{\eps t}{d}
\\
   && \qquad \textrm{ and } \sigma \max_{j\not=i} \tV^j_{s} < 1 - \frac{\eps s}{d},
       0 \le s\le t \Big)
\\
    &\geq& \P(\ttau_*=\ttau^i_*\le k_r) -  \P(\ttau_*\le \zeta+\ustar)
    \\
		&\geq&  \P(\ttau_*=\ttau^i_*)-  \P(\ttau^i_*> k_r) -  \P(\ttau_*\le \zeta+\ustar)
\end{eqnarray*}
where
\[
  \sigma_{*i}:=\sigma \min_{0\le t\le k_r}\frac{ 1 - \frac{\eps t}{d}}{1 - \frac{\eps (t-2\ustar)}{d}}
  = 1 - \frac{ \frac{2 \eps \ustar}{d}} {1 - \frac{\eps k_r}{d} + \frac{2\eps \ustar}{d}}.
\]
By the definition of $k_r$ we have
\[
  1 - \frac{\eps k_r}{d} = \frac{\gamma\sigma\sqrt{\psi}}{d} - \frac{r\sigma}{d\sqrt{\psi}}
  = \frac{\gamma\sigma\sqrt{\psi}}{d} \, (1+o(1)).
\]
It follows that
\[
 \frac{ \sigma_{*i}}{ \sigma} = 1 - \frac{2 \ustar \eps}{\gamma\sigma\sqrt{\psi}}(1+o(1))
 =  1 - \frac {2 \ustar}{\gamma\sqrt{\psi}}\, h\,(1+o(1))
 = 1+ o\left(\frac{1}{\psi}\right),
\]
as required in the assumption of Remark~\ref{rem:moderatesigmastar}
to Proposition~\ref{prop:moderatelyslowpullingstat}.

Now \eqref{eqn:moderatefacescaling}, \eqref{eqn:moderateexittimescaling_i2},
and \eqref{eqn:zetasmall} yield
\begin{eqnarray*}
  && \liminf_{\eps,\sigma\to 0} \P(\tau=\tau^i)
\\
  &\geq& \lim_{\eps,\sigma\to 0} \P(\ttau_*=\ttau^i_*)
	\\
	&& -  \lim_{\eps,\sigma\to 0} \P(\ttau^i_*> k_r)
   -\lim_{\eps,\sigma\to 0} \P(\tau\le \zeta) -\lim_{\eps,\sigma\to 0} \P(\ttau_*\le \zeta+\ustar) - \P(\Xi>F)
\\
  &=& p_i -  \exp\left( -a_i e^{br} \right) - \P(\Xi>F).
\end{eqnarray*}
Letting $r, F \to +\infty$, we obtain \eqref{eqn:liminf_face}.


\bigskip
{\bf Acknowledgement.} This research was supported by the co-ordinated grants of DFG (GO420/6-1)
and St.\,Petersburg State University (6.65.37.2017).

\end{document}